\documentclass{amsart}
\usepackage{dsfont, amssymb, mathtools, IEEEtrantools, subcaption}
\usepackage{calc,changepage}
\usepackage{csquotes}
\usepackage{hyperref}
\usepackage{enumitem}

\usepackage{xcolor}
\definecolor{wb}{RGB}{51,153,255}
\usepackage{tikz}
\usetikzlibrary{calc}
\usepackage{tikz-cd}

\usepackage[parfill]{parskip}
\numberwithin{equation}{subsection}
\usepackage{amsrefs}

\newcommand{\defeq}{\vcentcolon=}
\newcommand{\eqdef}{=\vcentcolon}

\makeatletter
\def\moverlay{\mathpalette\mov@rlay}
\def\mov@rlay#1#2{\leavevmode\vtop{%
   \baselineskip\z@skip \lineskiplimit-\maxdimen
   \ialign{\hfil$\m@th#1##$\hfil\cr#2\crcr}}}
\newcommand{\charfusion}[3][\mathord]{
    #1{\ifx#1\mathop\vphantom{#2}\fi
        \mathpalette\mov@rlay{#2\cr#3}
      }
    \ifx#1\mathop\expandafter\displaylimits\fi}
\makeatother

\newcommand{\longhookrightarrow}{\lhook\joinrel\longrightarrow}

\newtheoremstyle{definitions}
 	{\topsep}
	{\topsep}
	{}
	{}
	{\bfseries}
	{:}
	{.5em}
	{}
  
\newtheoremstyle{lemmata}
	{\topsep}
	{\topsep}
	{\itshape} 
	{}
	{\bfseries}
	{:}
	{.5em}
	{}

\theoremstyle{lemmata}
\newtheorem{Theorem}[subsection]{Theorem}
\newtheorem{Lemma}[subsection]{Lemma}
\newtheorem{Corollary}[subsection]{Corollary}
\newtheorem{Proposition}[subsection]{Proposition}

\theoremstyle{definitions}
\newtheorem{Definition}[subsection]{Definition}
\newtheorem{Remark}[subsection]{Remark}

\newtheorem*{Remarks-nn}{Remarks}

\newtheorem{Example}[subsection]{Example}
\newtheorem{Procedure}[subsection]{Procedure}
\newtheorem{Table}[subsubsection]{Table}

\DeclareMathOperator{\spec}{spec}

\DeclareMathOperator{\GL}{GL}

\DeclareMathOperator{\End}{End}
\DeclareMathOperator{\PGL}{PGL}
\DeclareMathOperator{\aut}{aut}
\DeclareMathOperator{\Mod}{Mod}
\DeclareMathOperator{\st}{st}

\title[The behavior of distinguished forms on the fundamental domain]{On Drinfeld modular forms of higher rank V: The behavior of distinguished forms on the fundamental domain}
\author{Ernst-Ulrich Gekeler}
\date{\today}

\newcommand*\rows{5}

\begin{document}

\begin{abstract}
	This paper continues work of the earlier articles with the same title. For two classes of modular forms $f$:
	
	\begin{itemize}
		\item para-Eisenstein series $\alpha_{k}$ and 
		\item coefficient forms ${}_a \ell_{k}$, where $k \in \mathds{N}$ and $a$ is a non-constant element of $\mathds{F}_{q}[T]$,
	\end{itemize}
	the growth behavior on the fundamental domain and the zero loci $\Omega(f)$
	as well as their images $\mathcal{BT}(f)$ in the Bruhat-Tits building $\mathcal{BT}$ are studied. We obtain a complete description for $f = \alpha_{k}$ and for those of the forms
	${}_{a}\ell_{k}$ where $k \leq \deg a$. It turns out that in these cases, $\alpha_{k}$ and ${}_{a}\ell_{k}$ are strongly related, e.g., $\mathcal{BT}({}_{a}\ell_{k}) = \mathcal{BT}(\alpha_{k})$, and that
	$\mathcal{BT}(\alpha_{k})$ is the set of $\mathds{Q}$-points of a full subcomplex of $\mathcal{BT}$ with nice properties. As a case study, we present in detail the outcome for the forms
	$\alpha_{2}$ in rank 3.
\end{abstract}

\maketitle

\setcounter{section}{-1}

\section{Introduction}

Drinfeld modular forms are a crucial ingredient in the arithmetic of global function fields $K$, e.g., $K = \mathds{F}_{q}(T)$. While the case of forms of rank 2, which shows striking similarities (but also
important differences) with the case of classical elliptic modular forms, has been developed since the late 1970's (e.g., \cite{Goss1980}, \cite{Gekeler1979}), serious work on higher rank Drinfeld modular
forms started only in the last years with papers of Basson, Breuer, Pink \cite{Basson2017} \cite{BassonBreuer2017} \cite{BassonBreuerPink1-ta} \cite{BassonBreuerPink2-ta} \cite{BassonBreuerPink3-ta}
and the author \cite{Gekeler2017} \cite{Gekeler-ta-1} \cite{Gekeler2018} \cite{Gekeler-ta-2}. 

In \cite{Gekeler2017} and \cite{Gekeler-ta-1}, properties of the building map $\lambda$ from the Drinfeld symmetric space $\Omega^{r}$ (the habitat of modular forms) to the Bruhat-Tits building 
$\mathcal{BT}^{r}$ were investigated and used to describe growth/decay properties of modular forms $f$ on the fundamental domain $\mathbf{F}^{r} \subset \Omega^{r}$ of the modular group
$\Gamma = \GL(r, \mathds{F}_{q}[T])$, and the zero locus of $f$.

If $f$ is an (ortho-) Eisenstein series $E_{k}$ or one of the basic coefficient forms $g_{1}, \dots, g_{r} = \Delta$, the results are rather complete, see e.g. \cite{Gekeler2017} Theorem 5.5 for the 
discriminant function $\Delta$.

The aim of the present work is to find corresponding descriptions for other distinguished modular forms for $\Gamma$ like the para-Eisenstein series $\alpha_{k}$ and the coefficient forms ${}_{a}\ell_{k}$.
The ${}_{a}\ell_{k}$ ($a \in \mathds{F}_{q}[T]$, $k \in \mathds{N}$) are the coefficients of the generic Drinfeld module $\phi^{\boldsymbol{\omega}}$ $(\boldsymbol{\omega} \in \Omega^{r}$) of rank $r$,
i.e., of the $a$-th operator polynomial $\phi_{a}^{\boldsymbol{\omega}} = a + \sum_{k \geq 1} {}_{a}\ell_{k} \tau^{k}$, in particular, $g_{k} = {}_{T}\ell_{k} $. Both the ortho- and the para-Eisenstein series are
analogues of---different aspects of---Eisenstein series in the theory of elliptic modular forms. Besides the intrinsic interest in the $\alpha_{k}$ (see \cite{ChenLee2012}, \cite{ChenLee2013-1}), these are
important as they approximate (after suitable normalization, see \cite{Gekeler-ta-1} Theorem 4.13) ${}_{a}\ell_{k}$ for $\deg a \gg 0$. The ${}_{a}\ell_{k}$ themselves play a role in height and isogeny estimates
(\cite{ChenLee2013-2}, \cite{ChenLee2019}, \cite{BreuerPazukiRazafinjatovo-ta}) and determine arithmetic-geometric properties of moduli schemes (\cite{Gekeler2019}, \cite{PapikianWei-ta}).

It has been observed in \cite{Gekeler-ta-1} (4.15) that the distinguished modular forms $f = E_{k}$, $g_{k}$, $\alpha_{k}$, ${}_{a}\ell_{k}$ have the following property (proven for $E_{k}$, $g_{k}$, conjectured for
$\alpha_{k}$, ${}_{a}\ell_{k}$): The image $\mathcal{BT}^{r}(f)$ in $\mathcal{BT}^{r}$ of the vanishing set $\Omega^{r}(f)$ is the set of $\mathds{Q}$-points of a full subcomplex of $\mathcal{BT}^{r}$ which is
everywhere of dimension $r-2$ and connected if $r \geq 3$.

This is astonishing as it is easy to find modular forms $f$ such that $\mathcal{BT}^{r}(f)$ fails to be (the set of $\mathds{Q}$-points of) a subcomplex, let alone of such a nice shape. We call modular
forms with that property \textbf{simplicial}.

For modular forms $f$ without zeroes on $\Omega^{r}$ (e.g., the discriminant function $\Delta$), the van der Put transform $P(f)$ provides a relationship between properties of $\lvert f \rvert$ and
certain functions on the set of arrows ($=$ oriented $1$-simplices) of $\mathcal{BT}^{r}$ \cite{Gekeler-ta-3}. That relationship may be extended to modular forms $f$ with zeroes, provided that $f$
is simplicial. The absolute value $\lvert f \rvert$ must then be replaced with the spectral norm $\lVert f \rVert_{\mathbf{x}}$, which in contrast to $\lvert f \rvert$ is still well-defined as a function on
$\mathcal{BT}^{r}(\mathds{Q})$. Instead of the source condition (valid for invertible $f$)
\[
	\sum_{e \in \mathbf{A}_{\mathbf{x},1}} P(f)(e) = 0,
\]
where the sum is over the $1$-arrows with origin the vertex $\mathbf{x}$ of $\mathcal{BT}^{r}$, that sum takes as value the \textbf{local inner degree} $N_{\mathbf{x}}(f)$, which in case $r=2$ is simply
the number of zeroes, counted with multiplicity, of $f$ in the fiber $\Omega_{\mathbf{x}}^{r} = \lambda^{-1}(\mathbf{x}) \subset \Omega^{r}$. Hence $\lVert f \rVert_{\mathbf{x}}$, $P(f)$ and $N_{\mathbf{x}}(f)$
are important attributes of the simplicial modular form $f$. Our main results are:
\begin{itemize}
	\item Theorem \ref{Theorem.Characterisation-of-certain-subsets-of-BTIQ}, which describes $\mathcal{BT}^{r}(\alpha_{k})$ and shows that $\alpha_{k}$ is simplicial in the above sense;
	\item Theorem \ref{Theorem.Characterisation-of-elements-of-BTIQ}, which gives an easy-to-verify criterion for the existence of zeroes of ${}_{a}\ell_{k}$ on fibers $\Omega_{\mathbf{x}}$ of $\lambda$;
	\item Theorem \ref{Theorem.On-simpliciality-of-modular-forms}. It states that for $k \leq \deg \alpha$ the forms ${}_{a}\ell_{k}$ and $\alpha_{k}$ share important properties. Their van der Put transforms $P(f)$, vanishing sets
	$\mathcal{BT}^{r}(f)$, and local inner degrees $N_{\mathbf{x}}(f)$ agree. In particular, such ${}_{a}\ell_{k}$ are simplicial.
\end{itemize}
The plan of the paper is as follows: After presenting some background and collecting prerequisites from \cite{Gekeler2017} and \cite{Gekeler-ta-1} in Section 1, we deal in Section 2
with the $p$-Eisenstein series $\alpha_{k}$. We first show (Theorem \ref{Theorem.Smoothness-of-certain-vanishing-locei}) that $\Omega^{r}(\alpha_{k})$ is always smooth (as is the case for the special $o$-Eisenstein series 
$E_{q^{k}-1}$ of weight $q^{k}-1$, \cite{Gekeler-ta-1} Theorem 4.5). Then we introduce the concept of \textbf{characteristic sequence} of the lattice $\Lambda_{\boldsymbol{\omega}}$ attached to
$\boldsymbol{\omega} \in \Omega^{r}$. It is a special form of arranging the information on the Newton polygon of the exponential function $e_{\boldsymbol{\omega}}$ corresponding to 
$\Lambda_{\boldsymbol{\omega}}$, and particularly well adapted to our study. From a careful analysis of the characteristic sequence, Theorem \ref{Theorem.Characterisation-of-certain-subsets-of-BTIQ} about $\mathcal{BT}^{r}(\alpha_{k})$ results.
Besides we derive several properties of the spectral norm $\lVert \alpha_{k} \rVert_{\mathbf{x}}$ as a function on $\mathcal{BT}^{r}(\mathds{Q})$.

Section 3 is devoted to a similar study of the coefficient forms ${}_{a}\ell_{k}$. We present in Theorem \ref{Theorem.Characterisation-of-elements-of-BTIQ} a criterion to decide whether or not some point $\mathbf{x} \in \mathcal{BT}^{r}(\mathds{Q})$
belongs to $\mathcal{BT}^{r}({}_{a}\ell_{k})$. It requires some explicit calculations on the fundamental domain and an adaptation of the characteristic sequence to the finite $\mathds{F}_{q}$-lattice
$_{a}\phi^{\boldsymbol{\omega}}$ of $a$-torsion points of $\phi^{\boldsymbol{\omega}}$. With the same techniques, we are able to show (under the assumption that $k$ doesn't exceed the degree of $a$)
that the van der Put transforms $P({}_{a}\ell_{k})$ and $P(\alpha_{k})$ agree (Theorem \ref{Theorem.On-simpliciality-of-modular-forms}). This also implies equality of $\mathcal{BT}^{r}({}_{a}\ell_{k})$ with $\mathcal{BT}^{r}(\alpha_{k})$, and of the local 
inner degrees. In particular ${}_{a}\ell_{k}$ is simplicial in this case (as is expected but not yet proved for all ${}_{a}\ell_{k}$).

Section 4 does not contribute new results in the proper sense; instead, we present in a case study all the details of the interplay of spectral norms, zero locus and local inner degrees of a simplicial modular form
$f$, for the form $f= \alpha_{2}$ in rank $r=3$. This is the most simple case next to rank 2 (settled in \cite{Gekeler1999-1} for $\alpha_{k}$ and in \cite{Gekeler1999-2} for the ${}_{a}\ell_{k}$); simple enough
to present in a few pages, but complex enough to display all the facets and techniques of the general case. We hope it will be useful for readers who want to get their hands on more involved examples.

As the present paper is part of an ongoing project, still many natural questions are open. Hopefully, these will be addressed in future work. We mention a few.
\begin{itemize}
	\item The smoothness of $\Omega^{r}({}_{a}\ell_{k})$ is established only for $k < r$ (\cite{Gekeler-ta-1} Theorem 4.19). Does it hold in general?
	\item Similarly, we know $\mathcal{BT}^{r}({}_{a}\ell_{k})$ (and, as a consequence, the simpliciality of ${}_{a}\ell_{k}$) only if $k$ is less or equal to the degree of $a$.
	\item After some normalization $\widetilde{(\,\cdot\,)}$, $\lim_{\deg a \to \infty} {}_{a}\tilde{\ell}_{k} = \tilde{\alpha}_{k}$ (\cite{Gekeler-ta-1}, Theorem 4.13). Estimate the error terms (say, above a vertex
	$\mathbf{x}$ of $\mathcal{BT}^{r}$)!
	\item How do the zero loci of ${}_{a}\ell_{k}$ and $\alpha_{k}$ in $\Omega_{\mathbf{x}}^{r}$ relate if $k \leq \deg a$ and $\mathbf{x} \in \mathcal{BT}^{r}(\alpha_{k}) = \mathcal{BT}^{r}({}_{a}\ell_{k})$?
	\item Determine the reductions $\bar{f}$ of the forms $f = E_{q^{k}-1}$, $\alpha_{k}$, ${}_{a}\ell_{k}$ (and notably of ${}_{T}\ell_{k} = g_{k}$) at vertices $\mathbf{x}$ of $\mathcal{BT}^{r}$
	(as we did for $\alpha_{2}$ in Section 4). This has been achieved in a handful of isolated cases (see \cite{Gekeler2017}, Section 7), but a general method lacks so far.
\end{itemize}

\subsection*{Notations and conventions} The notation agrees largely with that of \cite{Gekeler2017} and \cite{Gekeler-ta-1}:
\begin{itemize}
	\item $\mathds{F} = \mathds{F}_{q}$ ist the finite field with $q$ elements, of characteristic $p$, with algebraic closure $\bar{\mathds{F}}$;
	\item $A = \mathds{F}[T]$ is the polynomial ring over $\mathds{F}$ with field of fractions $K = \mathds{F}(T)$;
	\item $K_{\infty} = \mathds{F}((T^{-1}))$ is the completion of $K$ with respect to the absolute value $\lvert \cdot \rvert = \lvert \cdot \rvert_{\infty}$ at infinity, normalized by $\lvert T \rvert = q$;
	\item $C_{\infty}$ is the completed algebraic closure of $K_{\infty}$, $O_{\infty} \subset K_{\infty}$ and $O_{C_{\infty}} \subset C_{\infty}$ are the respective rings of integers, $\pi$ the uniformizer
	$T^{-1}$ of $O_{\infty}$. Note that the residue class field of $O_{C_{\infty}}$ equals the algebraic closure $\bar{\mathds{F}}$ of the residue class field $\mathds{F}$ of $O_{\infty}$;
	\item $\log = - \nu_{\infty} \colon C_{\infty}^{*} \to \mathds{Q}$ is the map $z \mapsto \log_{q} \lvert z \rvert$. 
\end{itemize}
Throughout, $r$ is a natural number larger or equal to 2, and $\Gamma = \GL(r,A)$ the modular group. All our modular forms are with respect to $\Gamma$.

As usual we identify the ring $\End_{\mathds{F}}( \mathds{G}_{a}/C_{\infty})$ of $\mathds{F}$-linear endomorphisms of the additive group $\mathds{G}_{a}/C_{\infty}$ (i.e., of $q$-additive polynomials
of shape $\sum a_{i}X^{q^{i}}$, where $a_{i} \in C_{\infty}$ and multiplication is defined through insertion) with the non-commutative polynomial ring $C_{\infty} \{ \tau \}$, where $\tau a = a^{q}\tau$
for constants $a$, through $X^{q^{i}} \leftrightarrow \tau^{i}$. 

Similarly, $q$-additive power series are identified with $C_{\infty}\{\{ \tau \}\}$. As long as convergence questions are neglected, we may replace $C_{\infty}$ with any $\mathds{F}$-algebra $R$,
so $\End_{\mathds{F}}(\mathds{G}_{a}/R) \cong R\{\tau\}$, etc.

An \textbf{$\mathds{F}$-lattice} in $C_{\infty}$ (\textbf{$A$-lattice} if it is an $A$-submodule) is a discrete $\mathds{F}$-subspace $\Lambda$ of $C_{\infty}$, of finite or infinite dimension.
Associated with $\Lambda$ there are the functions 
\begin{align}
	e_{\Lambda}(X) 		&= X \sideset{}{'}\prod_{\lambda \in \Lambda} (1 - X/\lambda) \label{Eq.Exponential-function-associated-with-IF-lattice}  \\
											&= \sum_{i \geq 0} \alpha_{i}(\Lambda) X^{q^{i}} = \sum \alpha_{i}(\Lambda) \tau^{i}	&& \text{(the \textbf{{\color{green} exponential function}})} \nonumber \\
	\log_{\Lambda}(X)	&= \sum_{i \geq 0} \beta_{i}(\Lambda) \tau^{i} \nonumber \\
											&= \text{inverse of $e_{\Lambda}(X)$ in $C_{\infty}\{\{ \tau \}\}$}	&&(\text{the \textbf{{\color{green} logarithm function}}}) \nonumber \\
	E_{k}(\Lambda)			&= \sideset{}{'}\sum_{\lambda \in \Lambda} \lambda^{-k}	&&\text{(the \textbf{{\color{green}$k$-th Eisenstein series})}}. \nonumber
\end{align}
Here and in the sequel, the primed product $\sideset{}{'}{\textstyle\prod}$ or sum $\sideset{}{'}{\textstyle\sum}$ is the product or sum over the non-vanishing elements of the index set.

Given $r$ elements $\omega_{1}, \dots, \omega_{r} \in C_{\infty}$ linearly independent over $K_{\infty}$ (\textbf{$K_{\infty}$-l.i.}), we write
\begin{equation} \label{Eq.A-lattice-generated-by-r-elements-of-C-infty}
	\Lambda = \Lambda_{\boldsymbol{\omega}} = \sum_{1 \leq i \leq r} A \omega_{i}
\end{equation}
for the $A$-lattice generated by $\boldsymbol{\omega} = (\omega_{1}, \dots, \omega_{r})$, $e_{\boldsymbol{\omega}} \defeq e_{\Lambda_{\boldsymbol{\omega}}}$ for its exponential function and
$\phi^{\boldsymbol{\omega}} = \phi^{\Lambda_{\boldsymbol{\omega}}}$ for the attached Drinfeld $A$-module of rank $r$. Its $a$-th operator polynomial ($a\in A$) is 
\begin{equation} \label{Eq.a-th-operator-polynomial}
	\phi_{a}^{\boldsymbol{\omega}}(X) = \sum_{0 \leq k \leq r \deg \alpha} {}_{a}\ell_{k}(\boldsymbol{\omega})X^{q^{i}}
\end{equation}
with kernel $_{a}\phi^{\boldsymbol{\omega}} \defeq \{ z \in C_{\infty} \mid \phi_{a}^{\boldsymbol{\omega}}(z)= 0\} \cong (A/(a))^{r}$. We write $E_{k}(\boldsymbol{\omega})$ for 
$E_{k}(\Lambda_{\boldsymbol{\omega}})$, $\alpha_{k}(\boldsymbol{\omega})$ for $\alpha_{k}(\Lambda_{\boldsymbol{\omega}})$, etc.

The \textbf{Drinfeld symmetric spaces} are
\begin{align}
					\Omega^{*} 	&= \Omega^{r,*} = \{ \boldsymbol{\omega} \in C_{\infty}^{r} \mid \omega_{1}, \dots, \omega_{r} \text{ $K_{\infty}$-l.i.}\} \label{Eq.Drinfeld-symmetrical-spaces} \\
	\text{and } \Omega 	&= \Omega^{r}	= \{ \boldsymbol{\omega} = (\omega_{1}: \ldots: \omega_{r}) \in \mathds{P}^{r-1}(C_{\infty}) \mid \omega_{1}, \dots, \omega_{r} \text{ $K_{\infty}$-l.i.}\}. \nonumber 
\end{align}
If not stated otherwise, we assume homogeneous coordinates on $\Omega$ normalized such that $\boldsymbol{\omega} = (\omega_{1}: \ldots : 1)$, i.e., $\omega_{r} = 1$.

If $\mathcal{S}$ is a simplicial complex, $\mathcal{S}(\mathds{Z})$ denotes the set of vertices, $\mathcal{S}(\mathds{R})$ the points of the realization and 
$\mathcal{S}(\mathds{Q}) \subset \mathcal{S}(\mathds{R})$ the set of points with rational barycentric coordinates. We often write $\mathcal{S}$ for $\mathcal{S}(\mathds{R})$, and 
\enquote{$\mathbf{x} \in \mathcal{S}$} will mean $\mathbf{x} \in \mathcal{S}(\mathds{R})$. A \textbf{full subcomplex} $\mathcal{T}$ of $\mathcal{S}$ is given by a subset $\mathcal{T}(\mathds{Z})$ of 
$\mathcal{S}(\mathds{Z})$; its simplices are those of $\mathcal{S}$ intersected with $\mathcal{T}(\mathds{Z})$.

The \textbf{Bruhat-Tits building} $\mathcal{BT} = \mathcal{BT}^{r}$ of $\PGL(r, K_{\infty})$ is a contractible simplicial complex of dimension $r-1$, on which $\PGL(r, K_{\infty})$ acts transitively. Its vertex
set $\mathcal{BT}(\mathds{Z})$ is the set of similarity classes $[L]$ of $O_{\infty}$-lattices in the $K_{\infty}$-vector space $V = K_{\infty}^{r}$. The vertices $v_{0} = [L_{0}], \dots, v_{s} = [L_{s}]$ form an $s$-simplex
if and only if the classes $[L_{i}]$ have representatives $L_{i}$ such that $L_{0} \supsetneq L_{1} \supsetneq \dots \supsetneq L_{s} \supsetneq \pi L_{0}$. The set
$\mathcal{BT}(\mathds{R})$ of the realization of $\mathcal{BT}$ is in 1-1 correspondence with the set of similarity classes of non-archimedean real-valued norms
$\nu$ on $V$, where the restriction to $\mathcal{BT}(\mathds{Z}) \subset \mathcal{BT}(\mathds{R})$ is given by 
\begin{align*}
	\mathcal{BT}(\mathds{Z})	&\longhookrightarrow \{ \text{classes of norms on $V$} \} \\
	[L]													&\longmapsto \text{class $[\nu_{L}]$ of $\nu_{L}$}
\end{align*}
with $\nu_{L}(\mathbf{v}) \defeq \inf \{ \lvert a \rvert^{-1} \mid a \in K_{\infty} \text{ such that $a\textbf{v} \in L$} \}$ for $\textbf{v} \in V$. $\Omega$ and $\mathcal{BT}$ are related through the 
\textbf{building map}
\begin{align}
	\lambda \colon \Omega 	&\longrightarrow \mathcal{BT}(\mathds{R}), \label{Eq.Building-map-between-Omega-and-Bruhat-Tits-building}\\
									\boldsymbol{\omega} 	&\longmapsto [\nu_{\boldsymbol{\omega}}] \nonumber ,
\end{align}
where $\nu_{\boldsymbol{\omega}}(\mathbf{v}) = \lvert \sum v_{i} \omega_{i} \rvert$ if $\boldsymbol{\omega} = (\omega_{1}: \ldots: 1)$ and $\mathbf{v} = (v_{1}, \ldots, v_{r})$. It is onto $\mathcal{BT}(\mathds{Q})$
and equivariant for the natural left actions of $\GL(r, K_{\infty})$ on $\Omega$ and $\mathcal{BT}(\mathds{R})$. 

Given $\mathbf{x} \in \mathcal{BT}(\mathds{Q})$, the inverse image
\begin{equation} \label{Eq.Admissible-affinoid-open-subspace-given-x}
	\Omega_{\mathbf{x}} \defeq \{ \boldsymbol{\omega} \in \Omega \mid \lambda(\boldsymbol{\omega}) = \mathbf{x} \}
\end{equation}
is an admissible affinoid open subspace of the analytic space $\Omega$, the structure of which is described in \cite{Gekeler-ta-1} Theorem 2.4. The spectral norm of a holomorphic function
$f$ on $\Omega_{\mathbf{x}}$ is defined by
\begin{equation}
	\lVert f \rVert_{\mathbf{x}} \defeq \sup_{\omega \in \Omega_{\mathbf{x}}} \lvert f(\boldsymbol{\omega}) \rvert.
\end{equation}
For the basic properties of all the objects listed and for further discussion, we refer to \cite{Gekeler-ta-1} and the references therein.

\section{Some background}

\subsection{The distinguished modular forms}

Given $k \in \mathds{N}_{0}$, the function $\alpha_{k} \colon \Omega \to C_{\infty}$, $\alpha_{k}(\boldsymbol{\omega}) = \alpha_{k}(\Lambda_{\boldsymbol{\omega}})$ (see 
\eqref{Eq.Exponential-function-associated-with-IF-lattice}) is holomorphic of weight $q^{k}-1$, that is, it satisfies 
\begin{equation}
	\alpha_{k}(\gamma \boldsymbol{\omega}) = \aut(\gamma, \boldsymbol{\omega})^{q^{k}-1} \alpha_{k}(\boldsymbol{\omega})
\end{equation}
for each $\gamma \in \Gamma = \GL(r,A)$, where
\begin{equation}
	\aut(\gamma, \boldsymbol{\omega}) = \sum_{1 \leq j \leq r} \gamma_{r,j} \omega_{j}.
\end{equation}
Together with the boundary condition it satisfies (\cite{Gekeler-ta-1}, 1.7, 1.8), this means that it is a modular form of weight $q^{k}-1$ (and type $0$) for $\Gamma$. Similarly, $\beta_{k} \colon \boldsymbol{\omega} \mapsto \beta_{k}(\Lambda_{\boldsymbol{\omega}})$ and for a given $a \in A$, the function ${}_{a}\ell_{k} \colon \boldsymbol{\omega} \mapsto {}_{a}\ell_{k}(\phi^{\boldsymbol{\omega}})$ are modular forms
of weight $q^{k}-1$ and 
\[
	E_{k} \colon \boldsymbol{\omega} \longmapsto E_{k}(\Lambda_{\boldsymbol{\omega}}) = \sideset{}{'}\sum_{a_{1}, \dots, a_{r} \in A} (a_{1}\omega_{1} + \dots + a_{r} \omega_{r})^{-k}
\]
is modular of weight $k$. By a well-known calculation, $\beta_{k} = -E_{q^{k}-1}$, where $E_{0}$ by definition equals ${-}1$. The forms $E_{k}$, $\alpha_{k}$, ${}_{a}\ell_{k}$ are the Eisenstein series,
para-Eisenstein series, coefficient forms of the respective weights. More specifically,
\begin{equation}
	g_{k} \defeq {}_{T}\ell_{k} \qquad (0 \leq k \leq r)
\end{equation}
are the \textbf{basic coefficient forms}, where $g_{0} = T$ and $g_{r}$ is also called the \textbf{Drinfeld discriminant form} of rank $r$. Specifying a Drinfeld $A$-module of rank $r$ over $C_{\infty}$
is the same as specifying the $r$-tuple $(g_{1}, \dots, g_{r})$ with the sole restriction that $\Delta = g_{r} \neq 0$. There are many well-known relations and recursions for these functions, which reflect
commutator relations in the non-commutative ring $C_{\infty}\{\{ \tau \}\}$. Some of these may be found, e.g., in \cite{Gekeler1988} Section 2.

\subsection{The fundamental domain} Let $\Lambda$ be an $A$-lattice of rank $r$ in $C_{\infty}$. A \textbf{successive minimum basis} (SMB) of $\Lambda$ is an ordered basis 
$\{\lambda_{1}, \dots, \lambda_{r} \}$ such that each $\lambda_{i}$ has minimal absolute value among $\Lambda \smallsetminus \sum_{1 \leq j < i} A\lambda_{j}$. Each lattice $\Lambda$ does possess an SMB
$\{ \lambda_{1}, \dots, \lambda_{r} \}$, and it has the crucial properties (\cite{Gekeler2019}, Proposition 3.1):
\subsubsection{}\label{Subsub.Orthogonal-basis-of-lattice}\stepcounter{equation}%
The $\lambda_{i}$ are \textbf{orthogonal}, that is, for coefficients $a_{i} \in K$
\[
	\Big\lvert \sum_{1 \leq i \leq r} a_{i} \lambda_{i} \Big\rvert = \sup_{i} \lvert a_{i} \rvert \lvert \lambda_{i} \rvert \qquad \text{holds;}
\]
\subsubsection{}\stepcounter{equation}%
The sequence $\lvert \lambda_{1} \rvert \leq \lvert \lambda_{2} \rvert \leq \dots \leq \lvert \lambda_{r} \rvert$ is an invariant of $\Lambda$ (that is, independent of the choice of the SMB), and is 
called the \textbf{$A$-spectrum} $\spec_{A}(\Lambda)$ of $\Lambda$.

Similarly, if $\Lambda$ is a (finite or infinite) $\mathds{F}$-lattice, an $\mathds{F}$-SMB is an ordered basis $\{ \lambda_{1}, \lambda_{2}, \dots \}$ such that each $\lambda_{i}$ has minimal absolute value among
$\Lambda \setminus \sum_{1 \leq j < i} \lambda_{j}$. Again, such an $\mathds{F}$-SMB always exists, and the \textbf{$\mathds{F}$-spectrum} $\spec_{\mathds{F}}(\Lambda) = ( \lvert \lambda_{1} \rvert \leq \lvert \lambda_{2} \rvert \leq \dots )$ is an invariant of $\Lambda$. The $\mathds{F}$-lattice is called \textbf{separable} if $\spec_{\mathds{F}}(\Lambda)$ is multiplicity-free, i.e., $\lvert \lambda_{1} \rvert < \lvert \lambda_{2} \rvert < \dots$ and \textbf{inseparable} if not, \textbf{$k$-inseparable} if $\lvert \lambda_{k} \rvert = \lvert \lambda_{k+1} \rvert$.

Consider the subset
\begin{equation}
	\mathbf{F} = \mathbf{F}^{r} = \{ \boldsymbol{\omega} \in \Omega \mid \{ \omega_{r}, \omega_{r-1}, \dots, \omega_{1} \} \text{ is an SMB of } \Lambda_{\boldsymbol{\omega}} = \sum A \omega_{i} \}.
\end{equation}\stepcounter{subsubsection}%
(Note the inverted order, so $\lvert \omega_{r} \rvert \leq \lvert \omega_{r-1} \rvert \leq \dots \leq \lvert \omega_{1} \rvert$.) It is the set of $C_{\infty}$-points of an admissible open subspace, labelled also
by $\mathbf{F}$, of the analytic space $\Omega$. The fact that each $A$-lattice $\Lambda$ of rank $r$ has an SMB implies that each $\boldsymbol{\omega} \in \Omega$ is $\Gamma$-equivalent
with at least one element (and, in fact, at most finitely many elements) of $\mathbf{F}$. We thus call $\mathbf{F}$ the \textbf{fundamental domain} for $\Gamma$ on $\Omega$.

\subsubsection{}\stepcounter{equation}%
The \textbf{standard apartment} $\mathcal{A}$ of $\mathcal{BT}$ is the full subcomplex defined by the standard torus of diagonal matrices of $\GL(r, K_{\infty})$, with set of vertices
\[
	\mathcal{A}(\mathds{Z}) = \{ [L_{\mathbf{n}}] \mid \mathbf{n} = (n_{1}, \dots, n_{r}) \in \mathds{Z}^{r} \},
\]	
where $[L_{\mathbf{n}}]$ is the similarity class of the $O_{\infty}$-lattice 
\[
	L_{\mathbf{n}} = (\pi^{n_{1}} O_{\infty}, \dots \pi^{n_{r}} O_{\infty}) \qquad \text{ in } V = K_{\infty}^{r}.
\]
We have $[L_{\mathbf{n}}] = [L_{\mathbf{n'}}]$ if and only if $\mathbf{n}' - \mathbf{n} = (n,n, \dots, n)$ for some $n \in \mathds{Z}$. The realization $\mathcal{A}(\mathds{R})$ (for which we briefly write
$\mathcal{A}$) is an euclidean affine space with translation group $\mathds{R}^{r}/\mathds{R}(1,1,\dots,1) \overset{\cong}{\to} \{ \mathbf{x} \in \mathds{R}^{r} \mid x_{r} = 0 \}$ and with the natural
choice of origin $\mathbf{0} = [L_{\mathbf{0}}]$. We use that isomorphism as a description for $\mathcal{A} = \mathcal{A}(\mathds{R})$. The choice of the standard Borel subgroup of 
upper triangular matrices determines the \textbf{standard Weyl chamber}
\begin{equation}
	\mathcal{W} = \{ \mathbf{x} \in \mathcal{A} \mid x_{i} \geq x_{i+1} \text{ for } 1 \leq i < r \}
\end{equation} \stepcounter{subsubsection}%
with  walls $\mathcal{W}_{i} = \{ \mathbf{x} \in \mathcal{W} \mid x_{i} = x_{i+1}\}$. $\mathcal{W}$ is a fundamental domain for the action of $\Gamma$ on $\mathcal{BT}$ in the classical sense: each 
$\mathbf{x} \in \mathcal{BT}(\mathds{R})$ has a unique representative modulo $\Gamma$ in $\mathcal{W}$.
\subsubsection{} \label{Subsub.The-set-WZ-is-monoid-freely-spanned}\stepcounter{equation}%
The set $\mathcal{W}(\mathds{Z})$ of vertices is the monoid freely spanned by the vectors
\[
	\mathbf{n}_{i} = (1,1,\dots,1,0,\dots,0) \qquad \text{($i$ ones, $(r-i)$-zeroes)},
\]
where $1 \leq i < r$. The order relation \enquote{$\prec$} on $\mathcal{W}(\mathds{Z})$ is defined by
\[
	\mathbf{n} \prec \mathbf{n}' \Longleftrightarrow \mathbf{n}' - \mathbf{n} \text{ is a non-negative integral combination of the $\mathbf{n}_{i}$}.
\]
The relationship with $\Omega$ and its fundamental domain is as follows:
\begin{equation}
	\lambda(\boldsymbol{F}) = \mathcal{W}(\mathds{Q}), \qquad \lambda^{-1}(\mathcal{W}) = \mathbf{F}.
\end{equation} \stepcounter{subsubsection}%
We define 
\[
	\mathbf{F}_{i} \defeq \lambda^{-1}(\mathcal{W}_{i}) = \{ \boldsymbol{\omega} \in \mathbf{F} \mid \lvert \omega_{i} \rvert = \lvert \omega_{i+1} \rvert \}
\]
and for $\mathbf{x} \in \mathcal{W}(\mathds{Q})$
\[
	\mathbf{F}_{\mathbf{x}} \defeq \lambda^{-1}(\mathbf{x}) = \{ \boldsymbol{\omega} \in \mathbf{F} \mid \log \omega_{i} = x_{i}, 1 \leq i \leq r \},
\]
i.e., $\mathbf{F}_{\mathbf{x}} = \Omega_{\mathbf{x}}$ as in \eqref{Eq.Admissible-affinoid-open-subspace-given-x}.
\subsection{} The author takes the opportunity to correct an annoying sign error from \cite{Gekeler2017} 2.2, repeated in \cite{Gekeler-ta-1} 2.1. In order to fit with the rest of these papers, the definitions
given there for the lattice $L_{\mathbf{k}}$ must be 
\[
	L_{\mathbf{k}} = (\pi^{k_{1}} O_{\infty}, \dots, \pi^{k_{r}} O_{\infty}) \qquad \text{as in (1.2.4)}
\]
but not $L_{\mathbf{k}} = (T^{k_{1}} O_{\infty}, \dots, T^{k_{r}} O_{\infty}) = (\pi^{-k_{1}} O_{\infty}, \dots, \pi^{-k_{r}} O_{\infty})$ as erroneously stated in \cite{Gekeler2017} and \cite{Gekeler-ta-1}.
\subsection{} In view of the fundamental domain property, it suffices to study the behavior of a modular form $f$ on $\mathbf{F}$. We define the zero loci of $f$ in $\Omega$ and $\mathbf{F}$ by
\[
	\Omega(f) = \{ \boldsymbol{\omega} \in \Omega \mid f(\boldsymbol{\omega}) = 0 \}, \qquad \mathbf{F}(f) = \Omega(f) \cap \mathbf{F},
\]
and their images in $\mathcal{BT}$:
\[
	\mathcal{BT}(f) = \lambda(\Omega(f)), \qquad \mathcal{A}(f) = \mathcal{BT}(f) \cap \mathcal{A}, \qquad \mathcal{W}(f) = \mathcal{BT}(f) \cap \mathcal{W}.
\]
Then $\mathcal{BT}(f) = \Gamma \mathcal{W}(f)$, and $\mathcal{W}(f)$ provides a coarse picture of the zero locus of $f$ on $\mathcal{F}$. The following definition is motivated from properties of 
many of the distinguished modular forms.

\begin{Definition} \label{Definition.Simplicial-modular-form}
	Let $f$ be a modular form for $\Gamma = \GL(r,A)$. It is \textbf{simplicial} if $\mathcal{BT}(f)$ is the set of $\mathds{Q}$-points of a full simplicial subcomplex of $\mathcal{BT}$ which is everywhere of
	dimension $r-2$. In this case, we also use \enquote{$\mathcal{BT}(f)$} to designate that complex.
\end{Definition}

\begin{Example}
	\begin{enumerate}[wide, label=(\roman*)]
		\item Suppose that $r=2$. Then \enquote{$f$ simplicial} merely means that $\mathcal{BT}(f)$ is contained in $\mathcal{BT}(\mathds{Z})$, the set of vertices of $\mathcal{BT}$, or
		$\mathcal{W}(f) \subset \mathcal{W}(\mathds{Z})$. Here $\mathcal{W}$ is a half-line %
		\begin{tikzpicture} %
				\draw[fill=black] (0,0) circle (1pt);  \draw[fill=black] (0.5,0) circle (1pt);  \draw[fill=black] (1,0) circle (1pt); \draw (0,0) -- (1.15,0);
		\end{tikzpicture} \dots 	In this case, all the distinguished modular forms are known to be simplicial: Eisenstein series $E_{k}$ \cite{Cornelissen1995} \cite{Gekeler1999-2}; para-Eisenstein series 
		$\alpha_{k}$ \cite{Gekeler1999-1}; coefficient forms ${}_{a}\ell_{k}$ \cite{Gekeler2011}.
		\item Let $r \geq 2$ be arbitrary. For an Eisenstein series $E_{k}$, $\mathcal{W}(E_{k}) = \mathcal{W}_{r-1}$ (\cite{Gekeler-ta-1} Theorem 4.5). Also, for the basic coefficient forms $g_{i}$ with 
		$1 \leq i < r$, $\mathcal{W}(g_{i}) = \mathcal{W}_{r-i}$ \cite{Gekeler-ta-1} Theorem 4.2). As $\mathcal{W}_{i}$ is a full subcomplex of $\mathcal{W}$ everywhere of dimension $r-2$, these forms 
		are simplicial. The same holds for $r=3$ and $\alpha_{k}$ if $k \leq 4$, as has been determined \enquote{by hand} in \cite{Gekeler-ta-1} 4.7. Later we will see (Theorem \ref{Theorem.Characterisation-of-certain-subsets-of-BTIQ}) that $\alpha_{k}$ is 
		always simplicial, as is the coefficient form ${}_{a}\ell_{k}$ if $k \leq \deg a$ (and presumably always).
	\end{enumerate}
\end{Example}

\subsection{}\label{Sub.Logarithm-of-invertible-holomorphic-function} Suppose for the moment that $u$ is an invertible (= nowhere vanishing) holomorphic function on $\Omega$. In \cite{Gekeler-ta-1} it is shown that 
$\lvert u \rvert$ is constant on the fibres $\Omega_{\mathbf{x}} = \lambda^{-1}(\mathbf{x})$ of the building map. Its logarithm $\log u = \log_{q} \lvert u \rvert$ may be regarded as a function on $\mathcal{BT}(\mathds{Q})$ 
and as such is affine, that is, it interpolates linearly in simplices. In \cite{Gekeler-ta-3} its \enquote{derivative} $P(u)$ is studied, and it is shown that $P(u)$ is a $\mathds{Z}$-valued harmonic $1$-cochain on $\mathcal{BT}$.
This means that $P(u)$ is a $\mathds{Z}$-valued function on the set $\mathbf{A}(\mathcal{BT})$ of \textbf{arrows} (= oriented $1$-simplices) of $\mathcal{BT}$ and satisfies certain relations
(A), (B), (C). It is defined by 
\begin{equation} \label{Eq.Definition-van-der-Put-transform}
	P(u)(( \mathbf{x}, \mathbf{y})) = \log_{q}( \lVert u \rVert_{\mathbf{y}} / \lVert u \rVert_{\mathbf{x}})
\end{equation} \stepcounter{subsubsection}%
and is called the \textbf{van der Put transform} of $u$. (A) is the trivial condition
\[
	\sum P(u)(e) = 0,
\]
if $e$ runs through the arrows of a closed path on $\mathcal{BT}$, (C) is irrelevant for the present purpose, and (B) is 
\[
	\sum_{e \in \mathbf{A}_{\mathbf{x},1}} P(u)(e) = 0,
\]
if $e = (\mathbf{x}, \mathbf{y})$ runs through the set $\mathbf{A}_{\mathbf{x},1}$ of arrows of type 1 emanating from a fixed vertex $\mathbf{x}$ of $\mathcal{BT}$. An arrow $e = (\mathbf{x}, \mathbf{y})$
is \textbf{of type 1} if it corresponds to $O_{\infty}$-lattices $L_{\mathbf{x}}$, $L_{\mathbf{y}}$ in $V = K_{\infty}^{r}$, where $L_{\mathbf{x}} \supset L_{\mathbf{y}} \supset \pi L_{\mathbf{x}}$ and
$\dim_{\mathds{F}}(L_{\mathbf{y}}/\pi L_{\mathbf{x}}) = 1$.

Now replace $u$ with an arbitrary simplicial modular form (possibly with zeroes) and define $P(u)$ still by the formula \eqref{Eq.Definition-van-der-Put-transform}. An adaptation of the proof of Theorem 2.6
in \cite{Gekeler-ta-1} yields that $\mathbf{x} \mapsto \log_{q} \lVert u \rVert_{\mathbf{x}}$ is still an affine function on $\mathcal{BT}(\mathds{Q})$ and that $P(u)$ again is $\mathds{Z}$-valued.
(Here the fact that $u$ doesn't vanish on $\lambda^{-1}(\overset{\circ}{\sigma})$ is crucial, where $\overset{\circ}{\sigma}$ is the interior of a simplex $\sigma$ of maximal dimension $r-1$.) Condition (A)
holds trivially also for the generalized definition of $P(u)$.

Suppose that $u$ is scaled such that $\lVert u \rVert_{\mathbf{x}} = 1$; then its reduction $\bar{u}$, a rational function on the canonical reduction $\bar{\Omega}_{\mathbf{x}}$ of $\Omega_{\mathbf{x}}$, is defined.
As is explained in \cite{Gekeler-ta-1} 2.3, $\bar{\Omega}_{\mathbf{x}}$ is isomorphic with $\mathds{P}^{r-1}/\mathds{F} \smallsetminus \bigcup \, H$, where $H$ runs through the finite set of hyperplanes defined
over $\mathds{F}$. Hence $\bar{u}$ determines a divisor $\operatorname{div}(\bar{u})$ on $\mathds{P}^{r-1}/\mathds{F}$.

\subsubsection{} \stepcounter{equation}%
Let $N_{\boldsymbol{x}}(u) \in \mathds{N}_{0}$ be the degree of the $\bar{\Omega}_{\mathbf{x}}$-part of $\operatorname{div}(\bar{u})$, i.e., of the part coprime with $\bigcup \, H$. (If, e.g., 
$r=2$ then $\bar{\Omega}_{\mathbf{x}}$ is isomorphic with $\mathds{P}^{1}/\mathds{F} \smallsetminus \mathds{P}^{1}(\mathds{F})$ and $N_{\mathbf{x}}(u)$ is the number of zeroes, counted with multiplicity, of
$\bar{u}$ on $\mathds{P}^{1}/\mathds{F} \smallsetminus \mathds{P}^{1}(\mathds{F})$.) We call $N_{\mathbf{x}}(u)$ the \textbf{local inner degree} of $u$ at $\mathbf{x}$. Now the proof of (B) sketched in
\cite{Gekeler-ta-3} 2.6 shows that condition (B) for simplicial modular forms reads as follows.

\begin{Proposition} \label{Proposition.Relation-local-inner-degree-van-der-Put-transform}
	Let $u$ be a simplicial modular form and $\mathbf{x}$ a vertex of $\mathcal{BT}$. Then for the above defined quantities $P(u)(e)$ and $N_{\mathbf{x}}(u)$, the condition 
	\begin{equation} \tag{B'}
		\sum_{e \in \mathbf{A}_{\mathbf{x},1}} P(u)(e) = N_{\mathbf{x}}(u)
	\end{equation}
	holds.
\end{Proposition}

In contrast with (B), the condition (C) mentioned in \ref{Sub.Logarithm-of-invertible-holomorphic-function} has no reasonable generalization to functions with zeroes, and is therefore omitted. Some non-trivial examples for (B') in the case where $r=2$
are given in \cite{Gekeler1999-1} Section 8 and \cite{Gekeler2011} Section 6. We will work out an example with $r=3$ in Section 4.

\subsection{} The investigation of zeroes of modular forms is governed by the following two basic principles.
\subsubsection{Vanishing principle}\label{Sub.Vanishing-principle} Let $f$ be a holomorphic function on $\Omega_{\mathbf{x}}$ for some $\mathbf{x} \in \mathcal{BT}(\mathds{Q})$. If $\lvert f \rvert$ is non-constant on $\Omega_{\mathbf{x}}$
then $f$ has a zero on $\Omega_{\mathbf{x}}$. 

This is a way of stating Theorem 2.4 in \cite{Gekeler-ta-1}.
\subsubsection{Spectral principle}\label{Subsub.Spectral-principle} Let a (finite or infinite) $\mathds{F}$-lattice $\Lambda$ be given. If $\alpha_{k}(\Lambda) = 0$ for some $k \in \mathds{N}$ then $\Lambda$ is $k$-inseparable (i.e., 
$\lvert \lambda_{k} \rvert = \lvert \lambda_{k+1} \rvert$ for some $\mathds{F}$-SMB $\{ \lambda_{1}, \lambda_{2}, \dots \}$ of $\Lambda$). Conversely, let $S$ be a finite or infinite subset of $\mathds{N}$.
If $\Lambda$ is $k$-inseparable for each $k \in S$ then there exists an isospectral lattice $\Lambda'$ (that is, $\spec_{\mathds{F}}(\Lambda) = \spec_{\mathds{F}}(\Lambda')$) such that
$\alpha_{k}(\Lambda') = 0$ for each $k \in S$. 

This is Proposition 1.11 of \cite{Gekeler-ta-1}.

These principles will be used to conclude (under some assumptions) the existence of zeroes of a modular form near a given $\boldsymbol{\omega} \in \Omega$.

\subsection{} Another important tool in our study are Moore determinants. Given $\omega_{1}, \dots, \omega_{n} \in C_{\infty}$, the \textbf{Moore determinant} $M(\omega_{1}, \dots, \omega_{r})$ is the 
determinant of the $(n\times n)$-matrix
\[
	\begin{pmatrix}
		\omega_{1}	& \omega_{1}^{q}	& \dots 		& \omega_{1}^{q^{n-1}} \\
		\vdots			& \vdots					& \ddots	& \vdots								\\
		\omega_{n}	& \omega_{n}^{q}	& \dots		& \omega_{n}^{q^{n-1}} 
	\end{pmatrix}.
\]
The relevant properties are (see \cite{Goss1996} Section 1, 1.3 or \cite{Ore1933}):
\begin{equation}
	M(\omega_{1}, \dots, \omega_{n}) \neq 0 \Longleftrightarrow \{ \omega_{1}, \dots, \omega_{n} \} \text{ is linearly independent over $\mathds{F}$}.
\end{equation}
Assume this is the case, and let $\Lambda$ be the $\mathds{F}$-lattice generated by the $\omega_{i}$, with exponential function $e_{\Lambda}$. Then
\begin{equation}
	e_{\Lambda}(X) = (-1)^{n} \frac{M(\omega_{1}, \dots, \omega_{n}, X)}{M(\omega_{1}, \dots, \omega_{n})^{q}}.
\end{equation}
That is, letting $M^{(i)}(\omega_{1}, \dots, \omega_{n})$ be the $n \times n$-Minor corresponding to $X^{q^{i}}$ ($0 \leq i \leq n$) in the $(n+1)\times(n+1)$-matrix for $M(\omega_{1}, \dots, \omega_{n},X)$, then
$M^{(0)}(\omega_{1}, \dots, \omega_{n}) = M(\omega_{1}, \dots, \omega_{n})^{q}$, $M^{(n)}(\omega_{1}, \dots, \omega_{n}) = M(\omega_{1}, \dots, \omega_{n})$, and for the coefficients of $e_{\Lambda}$,
\begin{equation} \label{Eq.Coefficients-of-Moore-determinant}
	\alpha_{i}(\Lambda) = (-1)^{i} \frac{M^{(i)}(\omega_{1}, \dots, \omega_{n})}{M^{(0)}(\omega_{1}, \dots, \omega_{n})}
\end{equation}
holds. As a special case we find
\begin{equation}
	\alpha_{n}(\Lambda) = \sideset{}{'}\prod_{\lambda \in \Lambda} \lambda = (-1)^{n} M(\omega_{1}, \dots, \omega_{n})^{q-1}.
\end{equation}
\subsection{Special Eisenstein series}\label{Sub.Special-Eisenstein-series} From an analytical point of view, the most simple and easy-to-handle modular forms for $\Gamma$ (and in fact, the first ones seriously 
studied \cite{Goss1980}) are the Eisenstein series, in particular the special ones (i.e., with weight of shape $q^{j}-1$). The first $r$ special Eisenstein series $E_{q^{j}-1}$ ($1 \leq j \leq r$) generate the ring 
$\Mod^{0}(\Gamma)$ of modular forms of type 0. It has been shown in \cite{Gekeler-ta-1} Theorem 4.5 that
\begin{equation}
	\begin{split}
		&\mathcal{W}(E_{k}) = \mathcal{W}_{r-1}(\mathds{Q}) \qquad (0 < k \equiv 0 \pmod{q-1}; \quad \text{in the sequel}\\
		&\text{we abuse notation and often write briefly $\mathcal{W}(E_{k}) = \mathcal{W}_{r-1}$, etc.)}
	\end{split}
\end{equation}
and that
\begin{equation} \label{Eq.Smooth-F-for-all-natual-numbers}
	\mathbf{F}(E_{q^{j}-1}) \text{ is smooth for each $j \in \mathds{N}$}.
\end{equation}
Moreover, for each non-empty subset $S$ of $\{1,2,\dots,r-1\}$ the $\mathbf{F}(E_{q^{j}-1})$ with $j \in S$ intersect transversally, and the analytic space $\bigcap_{j \in S} \mathbf{F}(E_{q^{j}-1})$ is smooth of
dimension $r-1-\#(S)$.

\subsection{Basic coefficient forms}\label{Sub.Basic-coefficient-forms} Here we have results of similar strength. Theorem 4.2 of \cite{Gekeler-ta-1} states that
\begin{align}
			\mathcal{W}(g_{i})	&	= \mathcal{W}_{r-i} 
	\intertext{and: For each $\varnothing \neq S \subset \{1,2,\dots,r-1\}$,}  
			\lambda \Big( \bigcap_{i \in S} \mathbf{F}(g_{i}) \Big) 	&= \bigcap_{i \in S} \mathcal{W}_{r-i} \label{Eq.Intersection-of-lambda-Fgis}
\end{align}
and the analytic space $\bigcap_{i \in S} \mathbf{F}(g_{i})$ is smooth of dimension $r-1-\#(S)$. 

Furthermore, the growth (or rather decay) of $\lVert g_{i} \rVert_{\mathbf{x}}$ for $\mathbf{x} \in \mathcal{W}$ 
\enquote{moving to infinity} is given in \cite{Gekeler2017} Corollary 4.16 in conjunction with Proposition 4.10.

\begin{Remark}
	Note that the smoothness statements in \eqref{Eq.Smooth-F-for-all-natual-numbers} and \eqref{Eq.Intersection-of-lambda-Fgis} immediately turn over to the same statements for the $\Omega(f)$, where $f = E_{q^{j}-1}$ or $f = g_{i}$. Note also that $\lVert E_{k} \rVert_{\mathbf{x}}$
	is constant equal to 1 on $\mathcal{W}$, and in fact $\lvert E_{k}(\boldsymbol{\omega}) \rvert = \lVert E_{k} \rVert_{\mathbf{x}} = 1$ for $\boldsymbol{\omega} \in \mathbf{F}_{\mathbf{x}}$ if 
	$\mathbf{x} \notin \mathcal{W}(E_{k}) = \mathcal{W}_{r-1}$.
\end{Remark}

\section{Para-Eisenstein series}

\subsection{} Are there results similar to \ref{Sub.Special-Eisenstein-series} or \ref{Sub.Basic-coefficient-forms} for the para-Eisenstein series $\alpha_{k}$? First, Theorem 4.8 of \cite{Cornelissen1995} states that
\subsubsection{}\label{Subsub.Characterisation-of-BTIQ} $\mathbf{x} \in \mathcal{BT}(\mathds{Q})$ belongs to $\mathcal{BT}(\alpha_{k})$ if and only if for one (and thus for each) element 
$\boldsymbol{\omega} \in \Omega_{\mathbf{x}} = \lambda^{-1}(\mathbf{x})$ the lattice $\Lambda_{\boldsymbol{\omega}}$ is $k$-ins (brief for $k$-inseparable); and:
\subsubsection{}\label{Subsub.For-all-1-r-1-analytich-space-smooth} For each $\varnothing \neq S \subset \{ 1,2, \dots, r-1 \}$, the analytic space $\bigcap_{i \in S} \Omega(\alpha_{i})$ is smooth of dimension $r-1-\#(S)$. 

The $k$-inseparability of $\Lambda_{\boldsymbol{\omega}}$ depends only on the spectrum of $\Lambda_{\boldsymbol{\omega}}$, hence on $\lambda(\boldsymbol{\omega})$. The non-trivial part in \ref{Subsub.Characterisation-of-BTIQ} is
to show that $k$-inseparability implies the existence of a zero $\alpha_{k}$ in $\Omega_{\boldsymbol{x}}$. In contrast with \ref{Sub.Special-Eisenstein-series} and \ref{Sub.Basic-coefficient-forms}, $\mathcal{W}(\alpha_{k})$ has no simple description, 
and it is not immediate that it is the set of $\mathds{Q}$-points of a simplicial complex. To this question, see Theorem \ref{Theorem.Characterisation-of-certain-subsets-of-BTIQ}. First we show the following result, which could (and should!) 
have been shown in \cite{Gekeler-ta-1}.

\begin{Theorem} \label{Theorem.Smoothness-of-certain-vanishing-locei}
	For each $k \in \mathds{N}$ the vanishing locus $\Omega(\alpha_{k})$ is smooth.
\end{Theorem}

\begin{proof}
	\begin{enumerate}[wide, label=(\roman*)]
		\item We may assume $k \geq r$, as the result for $k < r$ is in \ref{Subsub.For-all-1-r-1-analytich-space-smooth}.
		\item Let $\Omega^{*} \subset C_{\infty}^{r}$ be the cone above $\Omega$, i.e., $\Omega^{*} = \{ (\omega_{1}, \ldots, \omega_{r}) \mid (\omega_{1}: \ldots : \omega_{r}) \in \Omega\}$. We regard a modular form
		$f$ of weight $k$ as a $\Gamma$-invariant homogeneous function of weight ${-}k$ on $\Omega^{*}$, that is 
		\[
			f( c\omega_{1}, \ldots, c\omega_{r} ) = c^{-k} f(\omega_{1}, \dots, \omega_{r}), \qquad c \in C_{\infty}^{*}.
		\]
		Thus in particular, $\frac{\partial}{\partial \omega_{r}}f$ is defined.
		\item From the characteristic equation 
		\[
			 e_{\boldsymbol{\omega}}(Tz) = \phi_{T}^{\boldsymbol{\omega}}\big(e_{\boldsymbol{\omega}}(z) \big)
		\]
		we get by comparing coefficients
		\[
			T^{q^{k}} \alpha_{k}(\boldsymbol{\omega}) = \sum_{0 \leq i \leq k} g_{i}(\boldsymbol{\omega}) \alpha_{k-i}^{q^{i}}(\boldsymbol{\omega})
		\]
		(where $g_{0} = T$ and $g_{i} = 0$ for $i > r$). That is,
		\begin{equation} \label{Eq.Characterisation-of-Tqk-T-alpha-k}
			[k] \alpha_{k} = \sum_{1 \leq i \leq r} g_{i} \alpha_{k-i}^{q^{i}},
		\end{equation}
		where $[k]$ is short for $T^{q^{k}}-T$.
		\item If $D$ is one of the operators $\frac{\partial}{\partial \omega_{r}}$ ($1 \leq i \leq r$), then
		\[
			[k] D(\alpha_{k}) = \sum D(g_{i}) \alpha_{k-i}^{q^{i}}
		\]
		and for $\mathbf{D} = (\frac{\partial}{\partial \omega_{1}}, \dots, \frac{\partial}{\partial \omega_{r}})^{t}$,
		\begin{equation} \label{Eq.Jacobian-of-Operators}
			[k]\mathbf{D}(\alpha_{k}) = \left( \frac{\partial g_{i}}{\partial \omega_{j}} \right)_{1 \leq i,j \leq r} (\alpha_{k-1}^{q}, \dots, \alpha_{k-r}^{q^{r}})^{t}.
		\end{equation}
		Here $(~)^{t}$ means transpose.
		\item Suppose that $\boldsymbol{\omega} \in \Omega$ is such that $\alpha_{k}(\boldsymbol{\omega}) = 0$ and $\frac{\partial}{\partial \omega_{j}} \alpha_{k}(\boldsymbol{\omega}) = 0$ for $1 \leq j < r$.
		By Euler's formula,
		\[
			\sum_{1 \leq j \leq r} \omega_{j} \frac{\partial}{\partial \omega_{j}} \alpha_{k}(\boldsymbol{\omega}) = (1- q^{k}) \alpha_{k}(\boldsymbol{\omega}),
		\]
		hence also $\frac{\partial}{\partial \omega_{r}} \alpha_{k}(\boldsymbol{\omega}) = 0$. That is, the left hand side of \eqref{Eq.Jacobian-of-Operators} vanishes.
		\item It has been shown in \cite{Gekeler-ta-1} Proposition 3.14 that the determinant $\det ( \frac{\partial g_{i}}{\partial \mu_{j}})_{1 \leq i,j \leq r}$ vanishes nowhere, where $\mu_{1}, \dots, \mu_{r}$ 
		are coordinates on the space $N^{r,*}$ in \cite{Gekeler-ta-1} 3.7. Now the canonical map
		\begin{align*}
			\Omega^{*} = \Omega^{r,*}											&\longrightarrow N^{r,*} = \Gamma(T) \setminus \Omega^{r,*}, \\
									(\omega_{1}, \dots, \omega_{r})			&\longmapsto (\mu_{1}, \dots, \mu_{r})
		\end{align*}
		is étale, thus $\det (\frac{\partial g_{i}}{\partial \omega_{j}})_{1 \leq i,j \leq r} \neq 0$, too. (The $\mu_{i}$ are the functions $\mu_{i}(\boldsymbol{\omega}) = e_{\boldsymbol{\omega}}(\omega_{i}/T)$,
		and $\Gamma(T)$ is the full congruence subgroup of level $T$.)
		\item We conclude from \eqref{Eq.Jacobian-of-Operators} that $\alpha_{k-1}(\boldsymbol{\omega}) = \dots = \alpha_{k-r}(\boldsymbol{\omega}) = 0$. Now the recursion \eqref{Eq.Characterisation-of-Tqk-T-alpha-k} applied to $\alpha_{k+1}$, $\alpha_{k+2}$, \dots implies that
		all these vanish at $\boldsymbol{\omega}$, which is absurd. Hence an $\boldsymbol{\omega}$ as in (v) cannot exist.
	\end{enumerate}
\end{proof}

\subsection{} The starting point for our investigation of $\Omega(\alpha_{k})$ and $\mathcal{BT}(\alpha_{k})$ is \ref{Subsub.Characterisation-of-BTIQ}. As the property of $\Lambda_{\boldsymbol{\omega}}$ to be $k$-ins depends only on 
$\mathbf{x} = \lambda(\boldsymbol{\omega})$, we define $\mathbf{x} \in \mathcal{BT}(\mathds{Q})$ to be $k$-ins if this is the case for one (= for all) $\boldsymbol{\omega} \in \Omega_{\mathbf{x}}$ and
write $\mathcal{BT}(k)$ respectively $\mathcal{A}(k)$ respectively $\mathcal{W}(k)$ for the set of those $\mathbf{x} \in \mathcal{BT}(\mathds{Q})$, $\mathcal{A}(\mathds{Q})$, $\mathcal{W}(\mathds{Q})$ 
which are $k$-ins.

\subsection{}\label{Sub.x-in-WIQ} In the following, we assume without restriction that $\mathbf{x} \in \mathcal{W}(\mathds{Q})$. Let $e_{1} = (1, 0,\dots,0)$, \dots, $e_{r} = (0,\dots,0,1)$ be the standard basis vectors of
$V = K_{\infty}^{r}$ and $\kappa_{\boldsymbol{\omega}}$ the isomorphism of $K_{\infty}$-vector spaces 
\begin{equation} \label{Eq.Isomorphism-of-K-infty-vector-spaces}
	\begin{split}
		\kappa_{\boldsymbol{\omega}} \colon \Lambda_{\boldsymbol{\omega}} \otimes K_{\infty}	&\overset{\cong}{\longrightarrow} V, \\
																																										\omega_{i}	&\longmapsto e_{i}.
	\end{split}
\end{equation}
It induces the norm $\nu_{\mathbf{x}} \defeq \nu_{\boldsymbol{\omega}}$ on $V$, $\nu_{\boldsymbol{\omega}}(\mathbf{v}) = \lvert \sum \omega_{i} v_{i} \rvert$ for $\mathbf{v} = (v_{1}, \dots, v_{r}) \in V$,
which depends only on $\lambda(\boldsymbol{\omega}) = \mathbf{x} = (x_{1}, \dots, x_{r})$, where $x_{i} = \log \omega_{i}$. An $\mathds{F}$-SMB of $\Lambda_{\boldsymbol{\omega}}$ may be constructed
by arranging the subset
\[
	B_{\boldsymbol{\omega}} \defeq \{ T^{s} \omega_{i} \mid s \in \mathds{N}_{0}, 1 \leq i \leq r \}
\]
of $\Lambda_{\boldsymbol{\omega}}$ in a suitable order. Let $B \defeq \kappa_{\boldsymbol{\omega}}(B_{\boldsymbol{\omega}}) = \{ T^{s}e_{i}\}$ be the corresponding subset of $\sum Ae_{i} \subseteq V$.
We define the following total order on $B$ (depending on $\mathbf{x}$):
\begin{equation} \label{Eq.Total-oder-on-B}
	T^{s}e_{i} \leq T^{s'}e_{i'} \vcentcolon \Longleftrightarrow \nu_{\mathbf{x}}(T^{s}e_{i}) < \nu_{\mathbf{x}}(T^{s'}_{e_{i'}}) \quad \text{or} \quad (\nu_{\mathbf{x}}(T^{s}e_{i}) = \nu_{\mathbf{x}}(T^{s'}e_{i'}) \text{ and } i > i')
\end{equation}
and arrange $B = \{ \lambda_{1}, \lambda_{2}, \dots \}$ according to this order. We call $(\lambda_{1}, \lambda_{2}, \dots)$ the \textbf{characteristic sequence of $\mathbf{x}$ in $V$}. It is a specific 
$\mathds{F}$-SMB on $(\sum Ae_{i}, \nu_{\mathbf{x}})$, whose pre-image under $\kappa_{\boldsymbol{\omega}}$ is an $\mathds{F}$-SMB on $\Lambda_{\boldsymbol{\omega}}$ also called the
\textbf{characteristic sequence of $\Lambda_{\boldsymbol{\omega}}$}. Hence $\mathbf{x}$ is $k$-ins. if and only if $\nu_{\mathbf{x}}(\lambda_{k}) = \nu_{\mathbf{x}}(\lambda_{k+1})$. Note that 
\begin{align*}
	\lambda_{1}		&= e_{r} \quad (\text{as $\mathbf{x} \in \mathcal{W}$}), \\
	\lambda_{2}	&= \begin{cases} e_{r-1},	&\text{if $x_{r-1} < 1$ (i.e., $\lvert \omega_{r-1} \rvert < q \lvert \omega_{r} \vert = q$)}, \\ Te_{r},	&\text{if $x_{r-1} \geq 1$, etc.} \end{cases}
\end{align*}
\subsection{} Since we shall make heavy use of the construction, we write it down in detail for $\mathbf{x} = \mathbf{n} \in \mathcal{W}(\mathds{Z})$. As in \ref{Subsub.The-set-WZ-is-monoid-freely-spanned}, let 
$\{ \mathbf{n}_{1}, \dots, \mathbf{n}_{r-1}\}$ be the standard basis of the monoid $\mathcal{W}(\mathds{Z})$, where $\mathbf{n}_{i} = (1,\dots,1,0,\dots,0)$ with $i$ 1's and $r-i$ zeroes. Each 
$\mathbf{n} \in \mathcal{W}(\mathds{Z})$ may uniquely be written
\begin{equation} \label{Eq.Linear-combination-of-element-of-WIZ}
	\mathbf{n} = (n_{1} - n_{2}) \mathbf{n}_{1} + (n_{2}- n_{3}) \mathbf{n}_{2} + \dots + (n_{r-1} - n_{r}) \mathbf{n}_{r-1}.
\end{equation}
(Note that $x_{r} = 0$ for all $\mathbf{x} \in \mathcal{W}$; nevertheless it is useful to dispose of this redundant quantity.) That is, the combinatorial distance $d(\mathbf{n}, \mathbf{0})$ from the origin
$\mathbf{0}$ equals 
\begin{equation}
	d(\mathbf{n}, \mathbf{0}) = (n_{1} - n_{2}) + \dots + (n_{r-1} - n_{r}) = n_{1}.
\end{equation}
We further define
\begin{multline} \label{Eq.Recursion-for-hes}
	h_{1} = h_{1}(\mathbf{n}) =	n_{r-1} - n_{r}, \qquad h_{2}	= h_{1} + 2(n_{r-2} - n_{r-1}), \ldots \\ 	h_{r-1}	= h_{r-2} + (r-1)(n_{1}-n_{2}), \qquad h \defeq  h_{r-1}.
\end{multline}
The first $h$ of the vectors $\lambda_{j}$ in \eqref{Eq.Total-oder-on-B} are
\begin{align} 
	T^{s}e_{r}																										&\quad (0 \leq s < h_{1}; \text{ i.e., $h_{1} = (n_{r-1} - n_{r})$ many}); \label{Eq.Powers-of-T} \\
	T^{n_{r-1}+s}e_{r}	, T^{s}e_{r-1}															&\quad (0 \leq s < n_{r-2} - n_{r-1}, 2(n_{r-2}-n_{r-1}) \text{ many}); \nonumber \\
	T^{n_{r-2}+s}e_{r}, T^{n_{r-2}-n_{r-1}+s}e_{r-1}, T^{s}e_{r-2}	&\quad (0 \leq s < n_{r-3} - n_{r-2}, 3(n_{r-3} - n_{r-2}) \text{ many}); \nonumber \\
	\vdots																											&\quad \nonumber \\
	T^{n_{2}+s}e_{r}, T^{n_{2}-n_{r-1}+s}e_{r-1}, \dots, T^{s}e_{2}	&\quad (0 \leq s < n_{1} - n_{2}, (r-1)(n_{1} - n_{2}) \text{ many}). \nonumber
\end{align}
These are arranged in $(n_{r-i} - n_{r-i+1})$ cycles each of length $i$ ($1 \leq i < r$), where in each $i$-cycle the vectors $e_{r}$, \dots, $e_{r-i+1}$ with suitable coefficients $T^{s}$ occur. From $j = h+1$ on, the
behavior is completely regular with cycles of length $r$:
\begin{equation} \label{Eq.Relation-between-lambdas-and-powers-of-T}
	\lambda_{h+1} = T^{n_{1}}e_{r}, \lambda_{h+2} = T^{n_{1}-n_{r-1}}e_{r-1}, \dots, \lambda_{h+r-1} = e_{1} \quad \text{and} \quad \lambda_{j+r} = T\lambda_{j}.
\end{equation}
We note that the norms $\nu_{\mathbf{x}}(\lambda_{j})$ are equal in each cycle but strictly grow from each cycle to the next. Hence:

\begin{Proposition} \label{Proposition.Characterisation-of-inseparable-points-of-WIZ}
	The point $\mathbf{n} \in \mathcal{W}(\mathds{Z})$ is $k$-inseparable if and only if $k$ is not the least index of the cycle in \eqref{Eq.Powers-of-T} or \eqref{Eq.Relation-between-lambdas-and-powers-of-T} to which $\lambda_{k}$ belongs. \hfill \mbox{$\square$}
\end{Proposition}	

\subsection{} If $\mathbf{x}$ fails to be in $\mathcal{W}(\mathds{Q})$ but still belongs to $\mathcal{A}(\mathds{Q})$ then \eqref{Eq.Isomorphism-of-K-infty-vector-spaces} and \eqref{Eq.Total-oder-on-B} still make sense, $\{e_{i} \mid 1 \leq i \leq r \}$ is still an 
orthogonal basis of $(V, \nu_{\mathbf{x}})$ and \eqref{Eq.Total-oder-on-B} produces an $\mathds{F}$-SMB of $(\sum Ae_{i}, \nu_{\mathbf{x}})$, but in the case $x_{1} \geq x_{2} \geq \dots \geq x_{r}$ no longer holds. In particular,
$\lambda_{1}$ not necessarily equals $e_{r}$.

Letting $W$ be the Weyl group of $\mathcal{A}$, which is isomorphic with the symmetric group $S_{r}$ and permutes the coordinates of $\mathcal{A}(\mathds{R}) = \mathds{R}^{r}/\mathds{R}(1,1,\dots,1)$,
then $\mathcal{A} = W\mathcal{W}$ and for each modular form $f$ for $\Gamma$,
\begin{equation}
	\mathcal{A}(f) = \mathcal{BT}(f) \cap \mathcal{A} = W\mathcal{W}(f),
\end{equation}
that is, for $f = \alpha_{k}$, $\mathcal{A}(k) = W\mathcal{W}(k)$. As the form $\alpha_{1}$ equals $(T^{q}-T)^{-1}g_{1}$, we find
\begin{equation}
	\mathcal{W}(1) = \mathcal{W}(g_{1}) = \mathcal{W}_{r-1} \qquad \text{(see (1.12.1))}.
\end{equation}
(Following our general convention, we briefly write \enquote{$\mathcal{W}_{r-1}$} for the set of $\mathds{Q}$-points $\mathcal{W}_{r-1}(\mathds{Q})$ of the full subcomplex $\mathcal{W}_{r-1}$ of $\mathcal{BT}$.) 

The crucial step in determining the higher $\mathcal{W}(k) = \mathcal{W}(\alpha_{k})$ is the recursion procedure given by the next result.

\begin{Proposition}
	Let $\mathbf{x} \in \mathcal{W}(\mathds{Q})$ and $\mathbf{x}' \defeq \mathbf{x} - \mathbf{n}_{r-1} \in \mathcal{A}(\mathds{Q})$ be given. For each $k \in \mathds{N}$ the equivalence 
	\[
		\mathbf{x} \in \mathcal{W}(k+1) \Longleftrightarrow \mathbf{x}' \in \mathcal{A}(k)
	\] 
	holds.
\end{Proposition}

\begin{proof}
	Let $\{\lambda_{1}, \lambda_{2}, \dots, \}$ and $\{ \lambda_{1}', \lambda_{2}', \dots \}$ be the characteristic sequences of $\mathbf{x}$ and $\mathbf{x}'$ in $V$ as in \eqref{Eq.Total-oder-on-B}. 
	
	Suppose first that \fbox{$\mathbf{x}' \in \mathcal{W}$}, which means that $x_{r-1}' \geq 0$, i.e., $x_{r-1} \geq 1$. Let $j \in \mathds{N}$ be such that $e_{r-1} = \lambda_{j+1}'$ (that is, $j - 1 \leq x_{r-1}' < j$).
	Then $\lambda_{s} = \lambda_{s}' = T^{s-1}e_{r}$ for $1 \leq s \leq j$, $\lambda_{j+1} = T^{j}e_{r}$, $\lambda_{j+2} = \lambda_{j+1}' = e_{r-1}$, and for $k > j$
	\begin{equation}
		\lambda_{k} = \begin{cases} \lambda_{k-1}',		&\text{if } \lambda_{k-1}' = T^{s}e_{i}, i<r \\ T\lambda'_{k-1},	&\text{if } \lambda_{k-1}' = T^{s}e_{r}. \end{cases}
	\end{equation}
	Hence for $k > j$ always $\nu_{\mathbf{x}}(\lambda_{k}) = q \nu_{\mathbf{x}'}(\lambda_{k-1}')$ holds. We get
	\begin{equation} \label{Eq.Characterisation-of-valuation-of-lambdas}
		\nu_{\mathbf{x}}(\lambda_{k+1}) = \nu_{\mathbf{x}}(\lambda_{k+2}) \Longleftrightarrow \nu_{\mathbf{x}'}(\lambda_{k}') = \nu_{\mathbf{x}'}(\lambda_{k+1}')
	\end{equation}
	at least for $k \geq j$. If however $k < j$ then $q \nu_{\mathbf{x}}(\lambda_{k+1}) = \nu_{\mathbf{x}}(\lambda_{k+2})$ and $q \nu_{\mathbf{x}'}(\lambda_{k}') = \nu_{\mathbf{x}'}(\lambda_{k+1}')$, so both
	equalities fail, and the assertion is shown in this case.
	
	Now suppose that \fbox{$\mathbf{x}' \notin \mathcal{W}$}. Then still $x_{1}' \geq x_{2}' \geq x_{3}' \geq \dots \geq x_{r-1}'$, but $-1 \leq x_{r-1}' < 0$. Let $j$ be such that $e_{r} = \lambda_{j}'$; we 
	have $j > 1$ and 
	\begin{align*}
			\begin{cases} j<r, 		&\text{if }x_{r-j+1}' < 0,  x_{r-j}' \geq 0, \\ j=r, 	&\text{if } x_{1}' < 0. \end{cases}
	\end{align*}	
	Then
	\begin{equation} \label{Eq.Lambdas-Lambda-Primes-and-Ts}
		\begin{IEEEeqnarraybox*}{rClCrClCrClCrClCrCl}
										&&					&\quad & \lambda_{1}' 	&=& e_{r-1},	 	& \quad 	& \lambda_{2}' 	&=& e_{r-2},	& \quad \dots \quad 	& \lambda_{j-1}' 	&=& e_{r-j+1},	& \quad  & \lambda_{j}' 			&=& e_{r} \\
			\lambda_{1}		&=&	e_{r},	&				& \lambda_{2}	&=& e_{r-1},		& 				&	\lambda_{3}	&=& e_{r-2},	& \dots								& \lambda_{j}			&=& e_{r-j+1},	& \quad  & \lambda_{j+1}		&=& Te_{r}
		\end{IEEEeqnarraybox*}
	\end{equation}
	and for $k > j$ always
	\begin{equation}
		\nu_{\mathbf{x}}(\lambda_{k}) = q \nu_{\mathbf{x}'}(\lambda_{k-1}') \quad \text{holds}.
	\end{equation}
	As before we get \eqref{Eq.Characterisation-of-valuation-of-lambdas} for $k \geq j$, while that equivalence for $k<j$ follows from comparing the two lines of \eqref{Eq.Lambdas-Lambda-Primes-and-Ts}.
\end{proof}

\subsection{}\label{Sub.Procedure-to-construct-the-Wk} By the preceding we obtain the following procedure to construct the $\mathcal{W}(k)$.
\begin{itemize}
	\item $\mathcal{W}(1) = \mathcal{W}_{r-1}$ ($=\mathcal{W}_{r-1}(\mathds{Q})$)
	\item For $k \geq 1$, $\mathcal{W}(k+1) = (W\mathcal{W}(k) + \mathbf{n}_{r-1}) \cap \mathcal{W}$.
\end{itemize}
Now $\mathcal{W}_{r-1}$ is a simplicial complex with the properties
\begin{enumerate}
	\item $\mathcal{W}_{r-1}$ is a full subcomplex of $\mathcal{W}$ and thus of $\mathcal{BT}$;
	\item $\mathcal{W}_{r-1}$ is everywhere of dimension $r-2$ (each vertex belongs to a simplex of maximal dimension $r-2$);
	\item $\mathcal{W}_{r-1}$ is connected.
\end{enumerate}
We will see that these properties essentially turn over to all the $\mathcal{W}(k)$ and $\mathcal{BT}(k)$. In particular, all the $\alpha_{k}$ are simplicial.

It is obvious from the procedure that $\mathcal{W}(k)$ is (the set of $\mathds{Q}$-points of) a subcomplex, which will be labelled by the same symbol. Hence also $\mathcal{A}(k)$ and $\mathcal{BT}(k)$
are subcomplexes.

\begin{Proposition}
	$\mathcal{BT}(k)$, $\mathcal{A}(k)$ and $\mathcal{W}(k)$ are full subcomplexes of $\mathcal{BT}$ and everywhere of dimension $r-2$.
\end{Proposition}

\begin{proof}
	The result for $\mathcal{BT}(k)$ follows from that for $\mathcal{A}(k)$, since each simplex of $\mathcal{BT}$ is contained in some apartment. We use induction on $k$, where the case $k=1$ is
	in \ref{Sub.Procedure-to-construct-the-Wk}. Suppose that both properties hold for $\mathcal{W}(k)$ and $\mathcal{A}(k)$. Let $\sigma$ be a simplex in $\mathcal{A}$ with vertices in 
	$\mathcal{A}(k+1) = W\mathcal{W}(k+1)$. There exists
	$w \in \mathcal{W}$ such that $\sigma \subset w\mathcal{W}$. Then $\sigma \subset w\mathcal{W}(k+1)(\mathds{Z}) = w(\mathcal{A}(k) + \mathbf{n}_{r-1} \cap \mathcal{W})(\mathds{Z})$. As both 
	$\mathcal{A}(k) + \mathbf{n}_{r-1}$ and $\mathcal{W}$ are full subcomplexes, $\sigma$ is a simplex in $w(\mathcal{A}(k) + \mathbf{n}_{r-1})$ and in $w\mathcal{W}$, hence in $\mathcal{A}(k+1)$. 
	This (and a similar argument for $\mathcal{W}(k+1)$) shows that $\mathcal{A}(k+1)$ and $\mathcal{W}(k+1)$ are full subcomplexes.
	
	Let now $w \mathbf{x}$ be a vertex of $\mathcal{A}(k+1) = W(\mathcal{A}(k) + \mathbf{n}_{r-1})$ with $w \in W$, $\mathbf{x} \in (\mathcal{A}(k) + \mathbf{n}_{r-1})(\mathds{Z})$. As $\mathcal{A}(k)$
	and thus $\mathcal{A}(k) + \mathbf{n}_{r-1}$ are everywhere of dimension $r-2$, there exists an $(r-2)$-simplex $\sigma$ of $\mathcal{A}(k) + \mathbf{n}_{r-1}$ with $\mathbf{x} \in \sigma$. Then $w \sigma$
	is an $(r-2)$-simplex in $\mathcal{A}(k+1)$ that encompasses $w \mathbf{x}$. Hence the assertion is shown for $\mathcal{A}(k+1)$.
	
	Now suppose that $\mathbf{x}$ is a vertex of $\mathcal{W}(k+1)$, and let as above $\sigma$ be an $(r-2)$-simplex of $\mathcal{A}(k+1)$ that contains $\mathbf{x}$. If $\sigma \subset W(\mathds{Z})$ then
	$\sigma$ is in fact a simplex in $\mathcal{W}(k+1)$, and we are ready. Otherwise, $\mathbf{x}$ lies at the boundary of $\mathcal{W}$ and there exists some $w \in \mathcal{W}$ fixing 
	$\mathbf{x}$ such that $\sigma \subset w^{-1}\mathcal{W}(\mathds{Z})$. Then $w \sigma$ is an $(r-2)$-simplex in $\mathcal{W}(k+1)$ that contains $\mathbf{x}$.
\end{proof}

We let $\sigma = \{ \mathbf{n}_{i} \mid 0 \leq i < r \}$ be the standard $(r-1)$-simplex in $\mathcal{A}$, with faces $\sigma^{(i)} = \sigma \smallsetminus \{ \mathbf{n}_{i} \}$. (Here $\mathbf{n}_{0} = \mathbf{0}$
is the origin.) Which of the $\mathbf{n}_{i}$ belong to $\mathcal{W}(k)$?

\begin{Proposition} \label{Proposition.Characterisation-of-Elements-of-Wk}
	For $0 \leq i < r$ the following hold:
	\begin{enumerate}[label=$\mathrm{(\roman*)}$]
		\item $\mathbf{n}_{i} \in \mathcal{W}(k) \Leftrightarrow k \not\equiv r-i \pmod{r}$;
		\item $\sigma^{(i)}$ is an $(r-2)$-simplex in $\mathcal{W}(k) \Leftrightarrow k \equiv r - i \pmod{r}$.
	\end{enumerate}
\end{Proposition}

\begin{proof}
	Clearly, (ii) is a consequence of (i) and the fullness of the subcomplex $\mathcal{W}(k)$.
	
	For $i=0$, $\mathbf{n} = \mathbf{n}_{0} = \mathbf{0}$, we have $h(\mathbf{0}) = 0$. Proposition \ref{Proposition.Characterisation-of-inseparable-points-of-WIZ} together with \eqref{Eq.Relation-between-lambdas-and-powers-of-T} shows that $\mathbf{0} \in \mathcal{W}(k)$ if and only if $k \not\equiv 0 \pmod{r}$. 
	Let now $i > 0$. For $\mathbf{n} = \mathbf{n}_{i}$, in the characteristic sequence $\lambda_{1}, \lambda_{2}, \dots$ there is one $(r-i)$-cycle followed by $r$-cycles only. Also, by Proposition \ref{Proposition.Characterisation-of-inseparable-points-of-WIZ}, 
	$\mathbf{n}_{i} \in \mathcal{W}(k)$ if and only if $k \not\equiv r-i \pmod{r}$.
\end{proof}

\begin{Proposition}
	If $r > 2$ the complex $\mathcal{W}(k)$ is connected for each $k \in \mathds{N}$. (If $r=2$, $\mathcal{W}(k)$ is a finite set of vertices and thus in general not connected.)
\end{Proposition}

\begin{proof}
	\begin{enumerate}[wide, label=(\roman*)]
		\item Given $\mathbf{n} \in \mathcal{W}(k)(\mathds{Z})$, we will construct a path in $\mathcal{W}(k)$ from $\mathbf{n}$ to some $\mathbf{n}_{i}$, which by the last proposition will give the result.
		\item Assume first that \fbox{$k > h = h(\mathbf{n})$} (see \eqref{Eq.Recursion-for-hes}). Write $k = h+jr+k_{0}$, where $j \in \mathds{N}_{0}$ and $0 < k_{0} < r$ (which is possible as $k \not\equiv h \pmod{r}$ by (\ref{Proposition.Characterisation-of-inseparable-points-of-WIZ})). Let $i$
		be such that $n_{i} > n_{i+1}$; then $\mathbf{n}_{i} \prec \mathbf{n}$, that is, $\mathbf{n}_{i}$ occurs in the presentation \eqref{Eq.Linear-combination-of-element-of-WIZ} of $\mathbf{n}$. Put 
		$\mathbf{n}' \defeq \mathbf{n} - \mathbf{n}_{i} \in \mathcal{W}(\mathds{Z})$; then $\mathbf{n}'$ is a neighbor of $\mathbf{n}$, $h' \defeq h(\mathbf{n}') = h -r+i$ and $k = h+jr+k_{0} = h' + (j+1)r + k_{0} -i$. 
		If $i \neq k_{0}$ then $k \not\equiv h' \pmod{r}$, $\mathbf{n}' \in \mathcal{W}(k)$ and $d(\mathbf{n}', \mathbf{0}) = d(\mathbf{n}, \mathbf{0}) - 1$, $h' < h$.
		Such lowering of $d(\mathbf{n}, \mathbf{0})$ and $h(\mathbf{n})$ in $\mathcal{W}(k)$ works as long as there is some $i \neq k_{0}$, $1 \leq i < r$ with $n_{i} > n_{i+1}$. 
		\item If this fails then
		\begin{equation} \label{Eq.If-this-fails-equation-for-n}
			\mathbf{n} = (n_{i} - n_{i+1}) \mathbf{n}_{i} \quad \text{with} \quad (n_{i} - n_{i+1}) \geq 2
		\end{equation}
		(if $n_{i} - n_{i+1} = 1$ then $\mathbf{n} = \mathbf{n}_{i}$ and we are ready) and $k_{0} = i$. In this case, replace $\mathbf{n}$ with $\mathbf{n}' \defeq \mathbf{n} - (\mathbf{n}_{i} - \mathbf{n}_{\ell})$ with
		some $\ell \neq i$, $1 \leq \ell < r$. Then:
		\begin{itemize}
			\item $\mathbf{n}' \in \mathcal{W}(\mathds{Z})$;
			\item $\mathbf{n}$ and $\mathbf{n}'$ are neighbors in $\mathcal{W}$;
			\item $h' = h(\mathbf{n}') = h+i-\ell$;
			\item $k = h+jr+i = h' +jr + \ell$, so $\mathbf{n}' \in \mathcal{W}(k)$, too;
			\item $d(\mathbf{n}',\mathbf{0}) = d(\mathbf{n}, \mathbf{0})$.
		\end{itemize}
		\item Now the bad case \eqref{Eq.If-this-fails-equation-for-n} does not hold for $\mathbf{n}'$ and we may continue our lowering procedure by replacing $\mathbf{n}'$ with $\mathbf{n}'' = \mathbf{n}' - \mathbf{n}_{i}$ for some $i$
		and $d(\mathbf{n}'', \mathbf{0}) < d(\mathbf{n}, \mathbf{0})$. This way we arrive after a finite number of steps at some $\mathbf{n} \in \mathcal{W}(k)$ where either $d(\mathbf{n}, \mathbf{0}) = 1$ 
		(in which case we are ready) or $k \leq h(\mathbf{n})$. Hence it suffices to treat that case.
		\item Assume \fbox{$k \leq h = h(\mathbf{n})$}. More precisely, let $k$ satisfy
		\[
			h_{i-1} < k \leq h_{i} \qquad \text{(notation of \eqref{Eq.Recursion-for-hes}; $h_{0} \defeq 0$)}.
		\]
		This implies that $n_{r-i} > n_{r-i+1}$, $\lambda_{k}$ belongs to an $i$-cycle, and $\mathbf{n}_{r-i} \prec \mathbf{n}$. The $i$-cycle to which $\lambda_{k}$ belongs looks
		\[
			(*e_{r}, *e_{r-1}, \dots, *e_{r-i+1}) = (\lambda_{a+1}, \dots, \lambda_{a+i}),
		\]
		where the $*$'s stand for suitable powers of $T$. Since $\mathbf{n} \in \mathcal{W}(k)$, $\lambda_{k}$ is one of the first $i-1$ entries, i.e., $a+1 \leq k < a + i$.
		\item Replace $\mathbf{n}$ with $\mathbf{n}' \defeq \mathbf{n} - \mathbf{n}_{r-i}$. As usual, primed data refer to $\mathbf{n}'$. The cycle structure \eqref{Eq.Powers-of-T} of the $\lambda_{j}'$ (for $\mathbf{n}'$)
		is obtained from that of the $\lambda_{j}$ (for $\mathbf{n}$) by omitting one $i$-cycle. More concretely, $\lambda_{j}' = \lambda_{j}$ for $j \leq a$, and the $(~)'$-cycle starting with $\lambda_{a+1}'$ is
		$(\lambda_{a+1}', \dots, \lambda_{a+i}') = (\lambda_{a+1}, \dots, \lambda_{a+i})$, if the common value $\log_{q} \nu_{\mathbf{n}}(\lambda_{a+1}) = \dots = \lambda_{q} \nu_{\mathbf{n}}(\lambda_{a+i})$
		is less than $n_{r-i}' = n_{r-i}-1$ and $(\lambda_{a+1}, \dots, \lambda_{a+i}, e_{r-i}, \dots)$, if $\log_{q} \nu_{\mathbf{n}}(\lambda_{a+i}) = n_{r-i}-1$. In any case, $\lambda_{k}'$ is not the last entry in its
		cycle, and so $\mathbf{n}' \in \mathcal{W}(k)$. Then $\mathbf{n}'$ is a neighbor of $\mathbf{n}$ in $\mathcal{W}(k)$ with $d(\mathbf{n}', \mathbf{0}) = d(\mathbf{n}, \mathbf{0}) -1$, and we are 
		done by induction. \qedhere
	\end{enumerate}
\end{proof}

We summarize what has been obtained.

\begin{Theorem} \label{Theorem.Characterisation-of-certain-subsets-of-BTIQ}
	Let $\Omega(\alpha_{k})$ and $\mathbf{F}(\alpha_{k})$ be the vanishing sets of the modular form $\alpha_{k}$ $(k\in \mathds{N})$ in $\Omega$ and $\mathbf{F}$, respectively, and let 
	$\lambda \colon \Omega \to \mathcal{BT}(\mathds{Q})$ be the building map. Define the following subsets of $\mathcal{BT}(\mathds{Q})$:
	\begin{align*}
		\mathcal{BT}(k) 	&\defeq \lambda(\Omega(\alpha_{k})) \\
		\mathcal{W}(k)		&\defeq \lambda(\mathbf{F}(\alpha_{k})) \\
		\mathcal{A}(k)		&\defeq W\mathcal{W}(k)
	\end{align*}
	with the Weyl group $W$ of the apartment $\mathcal{A}$. Then $\mathcal{BT}(\mathds{Q})$, $\mathcal{A}(k)$, $\mathcal{W}(k)$ are the sets of $\mathds{Q}$-points of subcomplexes of $\mathcal{BT}$, denoted
	by the same symbols. These subcomplexes of $\mathcal{BT}$ are
	\begin{enumerate}[label=$\mathrm{(\roman*)}$]
		\item full subcomplexes of $\mathcal{BT}$ (simplices of, e.g., $\mathcal{W}(k)$ are simplices of $\mathcal{BT}$ intersected with $\mathcal{W}(k)(\mathds{Z})$);
		\item everywhere of dimension $r-2$ (each vertex of, e.g., $\mathcal{W}(k)$ is contained in an $(r-2)$-simplex of $\mathcal{W}(k)$);
		\item connected if $r > 2$.
	\end{enumerate}
	In particular, the modular form $\alpha_{k}$ is simplicial as defined in \ref{Definition.Simplicial-modular-form}. These subcomplexes are related by $\mathcal{BT}(k) = \Gamma \mathcal{W}(k) = \Gamma \mathcal{A}(k)$,
	$\mathcal{A}(k) = W\mathcal{W}(k)$, $\mathcal{W}(k) = \mathcal{A}(k) \cap \mathcal{W}$, $\mathcal{A}(k+1) = W(\mathcal{A}(k) + \mathbf{n}_{r-1})$, where $\Gamma$ is the group $\GL(r,A)$ 
	and $\mathbf{n}_{r-1} = (1,\dots,1,0)$.
\end{Theorem}

\begin{proof}
	Everything has been established except for the connectedness of $\mathcal{A}(k)$ and $\mathcal{BT}(k)$ (which follows from the connectedness of $\mathcal{A}(k)$) in the case where $r > 2$.
	Now $\mathcal{A}(k) = \bigcup_{w \in W} w\mathcal{W}(k)$, and all the connected complexes $w\mathcal{W}(k)$ are joined through $\mathbf{0}$ (if $k \not\equiv 0 \pmod{r}$) or the 
	$\mathbf{n}_{i}$ ($1 \leq i < r$, otherwise).
\end{proof}

\begin{Remark} \label{Remark.Discussion-and-explanation-for-the-convenience-of-the-reader}
	As results from the description of $\mathcal{W}(k)$ and its pre-image in $\mathbf{F}$, the zero set $\mathbf{F}(\alpha_{k})$ is contained in 
	$\{ \boldsymbol{\omega} \in \mathbf{F} \mid \log \omega_{r-1} \leq k-1\}$. Hence $\alpha_{k}$ has no zeroes in 
	$\mathbf{F}_{(\omega_{r-1} > k-1)} \defeq \{ \boldsymbol{\omega} \in \mathbf{F} \mid \omega_{r-1} > k-1 \}$. It is easy to show (\cite{Gekeler-ta-1}, proof of Theorem 4.13) and follows also from the 
	procedure below that $\lvert \alpha_{k} \rvert$ is constant on $\mathbf{F}_{(\omega_{r-1} > k-1)}$ with value $\lvert \alpha_{k}(A) \rvert$, where $\log \alpha_{k}(A) = q(q^{k} -1)/(q-1) - kq^{k}$.
	
	For the reader's convenience we describe the spectral norm $\lVert \alpha_{k} \rVert_{\mathbf{n}}$ of $\alpha_{k}$ on the vertex $\mathbf{n}$ of $\mathcal{W}$. Recall that its logarithm 
	$\log_{q} \lVert \alpha_{k} \rVert_{\mathbf{n}}$ interpolates linearly in $\mathcal{W}(\mathds{Q})$. The proof is implicit in \cite{Gekeler-ta-1}, proof of Theorem 4.8, and is therefore omitted.
\end{Remark}

\begin{Procedure}[to determine $\lVert \alpha_{k} \rVert_{\mathbf{n}}$] \label{Procedure.To-determine-spectral-norm-of-alpha-k}
	Given $\mathbf{n} \in \mathcal{W}(\mathds{Z})$, let $\lambda_{1}, \lambda_{2}, \dots$ be its characteristic sequence in $V = K_{\infty}^{r}$. Put $c_{k} \defeq \log_{q} (\nu_{\mathbf{n}}(\lambda_{k}))$, where
	$\nu_{\mathbf{n}}$ is the norm on $V$ determined by $\mathbf{n}$. We have $c_{k+1} = c_{k}$ if $\lambda_{k}$, $\lambda_{k+1}$ are in the same cycle \eqref{Eq.Powers-of-T} and $c_{k+1} = c_{k}+1$ if $\lambda_{k+1}$
	starts a new cycle. Then
	\[
		\log_{q} \lVert \alpha_{k} \rVert_{\mathbf{n}} = -(q-1) \sum_{1 \leq j \leq k} q^{j-1} c_{j}.
	\]
\end{Procedure}

\begin{Corollary}
	$\lVert \alpha_{k} \rVert_{\mathbf{n}}$ has the following order properties:
	\begin{enumerate}[label=$\mathrm{(\roman*)}$]
		\item Fixing $\mathbf{n}$, $\lVert \alpha_{k} \rVert_{\mathbf{n}}$ decreases monotonically in $k$; strictly monotonically if $k > \sup \{ i \mid 1 \leq i < r \text{ and } n_{r-i+1} = 0 \}$;
		\item Fixing $k$, $\lVert \alpha_{k} \rVert_{\mathbf{n}}$ decreases monotonically in $\mathbf{n}$ with respect to the order \enquote{$\prec$} on $\mathcal{W}(\mathds{Z})$.
	\end{enumerate}
\end{Corollary}

\begin{proof}
	\begin{enumerate}[wide, label=(\roman*)]
		\item $c_{k} > 0$ if $k$ is larger than the supremum.
		\item Consider the characteristic sequences $\lambda_{1}, \lambda_{2}, \dots$ for $\mathbf{n}$ and $\lambda_{1}', \lambda_{2}', \dots$ for $\mathbf{n}' = \mathbf{n} + \mathbf{n}_{i}$ for some $i$.
		It is easily seen that always the equality $\nu_{\mathbf{n}'}(\lambda_{k}') \geq \nu_{\mathbf{n}}(\lambda_{k})$ holds.
	\end{enumerate}
\end{proof}

\begin{figure}[ht!]
	\centering
	\begin{tikzpicture}
	    \foreach \row in {0, 1, ...,\rows} {
	        \draw ($\row*(0.5, {0.5*sqrt(3)})$) -- ($(\rows,0)+\row*(-0.5, {0.5*sqrt(3)})$);
	        \draw ($\row*(1, 0)$) -- ($(\rows/2,{\rows/2*sqrt(3)})+\row*(0.5,{-0.5*sqrt(3)})$);
	        \draw ($\row*(1, 0)$) -- ($(0,0)+\row*(0.5,{0.5*sqrt(3)})$);
	    }
	    
	    \node[below] (0) at (0,0) {$\mathbf{0}$};
	    \node[below] (n1) at (1,0) {$\vphantom{2}\mathbf{n}_{1}$};
	    \node[below] (2n1) at (2,0) {$2\mathbf{n}_{1}$};
	    \node[left] (n2) at (0.5, {sqrt(3)/2}) {$\mathbf{n}_{2}$};
	    \node[left] (2n2) at (1, {sqrt(3)}) {$2\mathbf{n}_{2}$};
	    \node[above] (n1plusn2) at (1.5, {sqrt(3)/2}) {$\mathbf{n}_{1}+\mathbf{n}_{2}$};
	    
	    \draw (4,0) -- (6,0) node[below, pos=0.75] {$\mathcal{W}_{2}$};
	    \draw (4,0) -- ++(0.75,{0.75*sqrt(3)});
	    \draw (3.5, {sqrt(3)/2}) -- ++(0.75,{0.75*sqrt(3)});
	    \draw (3, {sqrt(3)}) -- ++(0.75,{0.75*sqrt(3)});
	    \draw (2.5, {3*sqrt(3)/2}) -- ++(0.75,{0.75*sqrt(3)});
	    \draw (2, {2*sqrt(3)}) -- ++(0.75,{0.75*sqrt(3)}) node[left, pos=0.9] {$\mathcal{W}_{1}$};
	\end{tikzpicture}
	\caption[The Weyl Chamber]{\emph{The Weyl Chamber}. Here and in Figure \ref{Figure2.Weyl-chambers-for-k-2-to-5} we present the $\mathcal{W}(k)$ for $r=3$ and $2 \leq k \leq 5$ as subsets of $\mathcal{W} \subset \mathcal{BT}^{3}$.} \label{Figure.The-Weyl-Chamber}
\end{figure}
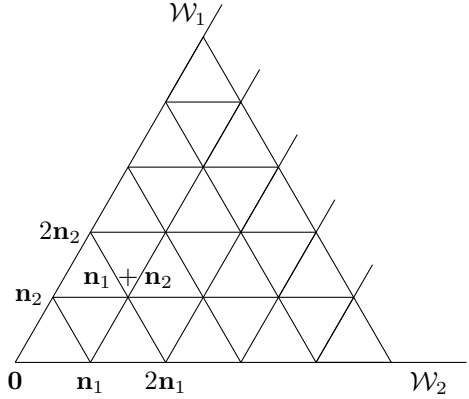
	
	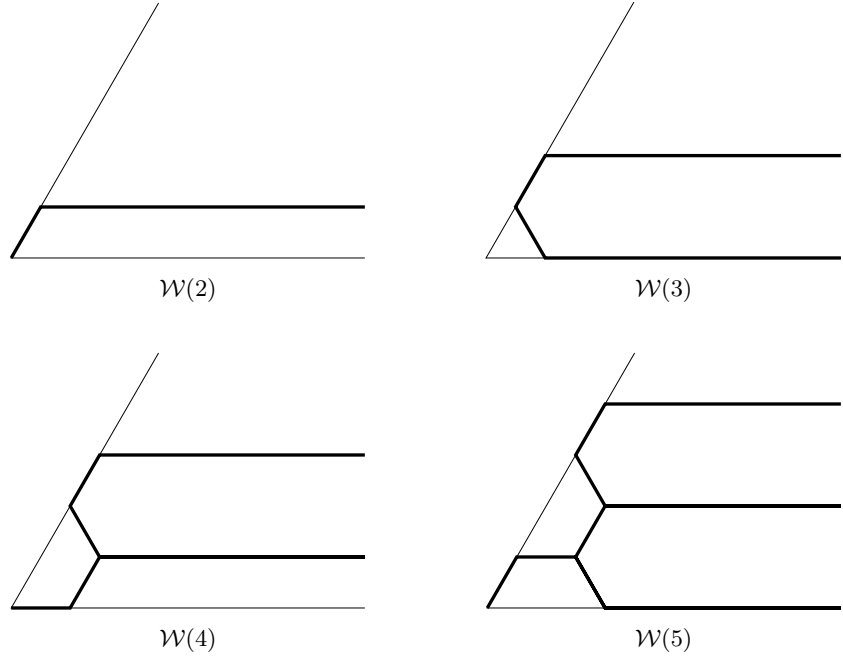
\begin{figure}[ht] 
		  \begin{subfigure}[b]{0.5\linewidth}
		    \centering
		     \resizebox{.75\textwidth}{!}{%
		    \begin{tikzpicture}
				\draw (6,0) -- (0,0) -- (2.5, {5*sqrt(3)/2});
				\draw[ultra thick] (0,0) -- (0.5,{sqrt(3)/2}) -- (6,{sqrt(3)/2});
			\end{tikzpicture}%
			}
		    \caption*{$\mathcal{W}(2)$} 
		    \label{Fig.W2} 
		    \vspace{4ex}
		  \end{subfigure}
		  \begin{subfigure}[b]{0.5\linewidth}
		    \centering
		     \resizebox{.75\textwidth}{!}{%
		   	\begin{tikzpicture}
				\draw (6,0) -- (0,0) -- (2.5, {5*sqrt(3)/2});
				\draw[ultra thick] (6,0) -- (1,0) -- (0.5,{sqrt(3)/2}) -- (1, {sqrt(3)}) -- (6,{sqrt(3)});
			\end{tikzpicture}%
			}
		    \caption*{$\mathcal{W}(3)$} 
		    \label{Fig.W3} 
		    \vspace{4ex}
		  \end{subfigure} 
		  \begin{subfigure}[b]{0.5\linewidth}
		    \centering
		     \resizebox{.75\textwidth}{!}{%
		    \begin{tikzpicture}
				\draw (6,0) -- (0,0) -- (2.5, {5*sqrt(3)/2});
				\draw[ultra thick] (0,0) -- (1,0) -- (1.5,{sqrt(3)/2}) -- (6,{sqrt(3)/2}) -- (1.5, {sqrt(3)/2}) -- (1,{sqrt(3)}) -- (1.5, {3*sqrt(3)/2}) -- (6, {3*sqrt(3)/2});
			\end{tikzpicture}%
			}
		    \caption*{$\mathcal{W}(4)$} 
		    \label{Fig.W4} 
		  \end{subfigure}
		  \begin{subfigure}[b]{0.5\linewidth}
		    \centering
			  \resizebox{.75\textwidth}{!}{%
			  		\begin{tikzpicture}
					\draw (6,0) -- (0,0) -- (2.5, {5*sqrt(3)/2});
					\draw[ultra thick] (0,0) -- (0.5, {sqrt(3)/2}) -- (1.5, {sqrt(3)/2}) -- (2,0) -- (6,0) -- (2,0) -- (1.5, {sqrt(3)/2}) -- (2, {sqrt(3)}) -- (6, {sqrt(3)}) -- (2, {sqrt(3)}) -- (1.5, {3*sqrt(3)/2}) -- (2, {2*sqrt(3)}) -- (6, {2*sqrt(3)});
				\end{tikzpicture}%
				}
		    \caption*{$\mathcal{W}(5)$} 
		    \label{Fig.W5} 
		  \end{subfigure} 
		  \caption{$\mathcal{W}(k)$, highlighted. See Figure \ref{Figure.The-Weyl-Chamber}.} \label{Figure2.Weyl-chambers-for-k-2-to-5}
		  \label{fig7} 
	\end{figure}

\section{The coefficient forms ${}_{a}\ell_{k}$}

\subsection{} In the whole section, $a$ is a fixed element of degree $d \geq 1$ of $A = \mathds{F}[T]$. The $a$-torsion $\{ z \in C_{\infty} \mid \phi_{a}(z) = 0\}$ is labelled by $_{a}\phi$. Here
$\phi = \phi^{\boldsymbol{\omega}}$ is the Drinfeld module of rank $r$ associated with $\boldsymbol{\omega} \in \Omega$. Then
\begin{equation} \label{Eq.Drinfeld.module-of-rank-r}
	\phi_{a}(X) = \sum_{0 \leq k \leq rd} {}_{a}\ell_{k}(\boldsymbol{\omega}) X^{q^{k}} = aX \sideset{}{'}\prod_{z \in _{a}\phi} (1 - X/z) = a e_{_{a}\phi}(X).
\end{equation}
That is, with notation as in \eqref{Eq.Exponential-function-associated-with-IF-lattice},
\begin{equation} \label{Eq.Coefficient-forms}
	{}_{a}\ell_{k} = \alpha_{k}(_{a}\phi).
\end{equation}
The vanishing of ${}_{a}\ell_{k}$ as a function in $\boldsymbol{\omega}$, or its potential vanishing in $\Omega_{\boldsymbol{x}}$ for $\boldsymbol{x} \in \mathcal{BT}(\mathds{Q})$, is therefore related with the
spectral properties of $_{a}\phi$. In this respect we have the following result, which is analogous with the corresponding statement for $\alpha_{k}$ (i.e., Theorem 4.8 in \cite{Gekeler-ta-1}, referred to in \ref{Subsub.Characterisation-of-BTIQ} and
\ref{Subsub.For-all-1-r-1-analytich-space-smooth}).

\begin{Theorem} \label{Theorem.Characterisation-of-elements-of-BTIQ}~
	\begin{enumerate}[label=$\mathrm{(\roman*)}$]
		\item For $\mathbf{x} \in \mathcal{BT}(\mathds{Q})$, the following are equivalent:
		\begin{enumerate}[label=$\mathrm{(\Alph*)}$]
			\item $\mathbf{x} \in \mathcal{BT}({}_{a}\ell_{k})$ (that is, there exists $\boldsymbol{\omega} \in \Omega_{\mathbf{x}} = \lambda^{-1}(\mathbf{x})$ such that ${}_{a}\ell_{k}(\boldsymbol{\omega}) = 0$);
			\item There exists $\boldsymbol{\omega} \in \Omega_{\mathbf{x}}$ such that the $\mathds{F}$-lattice $_{a}\phi^{\boldsymbol{\omega}}$ is $k$-inseparable;
			\item For each $\boldsymbol{\omega} \in \Omega_{\mathbf{x}}$, $_{a}\phi^{\boldsymbol{\omega}}$ is $k$-inseparable.
		\end{enumerate}
		\item For each nonempty subset $S$ of $\{1,2,\dots, r-1\}$, the analytic space $\bigcap_{i \in S} \Omega({}_a{\ell}_{i})$ is smooth of dimension $r-1 - \#(S)$.
	\end{enumerate}
\end{Theorem}

We note that (ii) has been shown in \cite{Gekeler-ta-1} Theorem 4.9. For the proof of (i), we need some preparations.
\subsection{} $_{a}\phi^{\boldsymbol{\omega}}$ is an $\mathds{F}$-vector space of dimension $rd$, with basis
\begin{equation}
	_{a}B_{\boldsymbol{\omega}} = \{ e_{\boldsymbol{\omega}}(T^{s}\omega_{i}/a) \mid 0 \leq s < d, 1 \leq i \leq r \},
\end{equation}
where $e_{\boldsymbol{\omega}} = e_{\Lambda_{\boldsymbol{\omega}}}$. Having fixed $a$ and $\boldsymbol{\omega}$, put for short 
\begin{equation}
	e_{s,i} \defeq e_{\boldsymbol{\omega}}(T^{s}\omega_{i}/a).
\end{equation}
In what follows, we assume that \fbox{$\boldsymbol{\omega} \in \mathbf{F}$}. Due to the characteristic property \ref{Subsub.Orthogonal-basis-of-lattice} of the SMB $\{\omega_{r}, \omega_{r-1}, \dots, \omega_{1}\}$ of 
$\Lambda_{\boldsymbol{\omega}}$, this eases the calculation of the $\lvert e_{s,i} \rvert$. 
\subsection{}\label{Sub.Product-representation-of-esi} Using the product representation
\[
	e_{s,i} = \frac{T^{s}\omega_{i}}{a} \sideset{}{'} \prod_{a_{1}, \dots, a_{r} \in A} \left( 1 - \frac{T^{s} \omega_{i}/a}{a_{1}\omega_{1} + \dots + a_{r} \omega_{r}} \right)
\]
we find
\begin{equation} \label{Eq.Absolute-value-of-esi}
	\lvert e_{s,i} \rvert = \lvert T^{s}\omega_{i}/a \rvert \sideset{}{'} \prod_{a_{i+1}, \dots, a_{r} \in A} \left| \frac{T^{s}\omega_{i}/a}{a_{i+1}\omega_{i+1} + \dots + a_{r}\omega_{r}} \right|,
\end{equation} \stepcounter{subsubsection}%
where the product $\sideset{}{'}{\textstyle\prod}$ is finite and over those $a_{i+1}, \dots, a_{r} \in A$ such that $\lvert \sum_{i < n \leq r} a_{n}\omega_{n} \rvert < \lvert T^{s} \omega_{i}/a \rvert$, i.e.,
$\lvert a_{n}\omega_{n} \rvert < \lvert T^{s} \omega_{i}/a \rvert$ for all $i < n \leq r$. Note that 
\subsubsection{} \stepcounter{equation}%
the product $\sideset{}{'}{\textstyle\prod}$ in \eqref{Eq.Absolute-value-of-esi}  may be empty, in which case it evaluates to 1, and
\begin{equation}
	\left\lvert 1 - \frac{T^{s}\omega_{i}/a}{a_{i+1}\omega_{i+1} + \dots + a_{r}\omega_{r}} \right\rvert = 1
\end{equation} \stepcounter{subsubsection}%
if the ratio on the right hand side has absolute value $1$.
\subsection{} Fix $s$ in \eqref{Eq.Absolute-value-of-esi}  and consider for $1 \leq i < r$ the ratio
\begin{equation}
	\lvert e_{s,i} / e_{s,i+1} \rvert = \left\lvert \frac{\omega_{i}}{\omega_{i+1}} \right\rvert \prod_{a_{i+1}, \dots, a_{r}} \frac{\lvert \omega_{i}/\omega_{i+1} \rvert}{\lvert a_{i+1}\omega_{i+1} + \dots + a_{r}\omega_{r} \rvert},
\end{equation}
where $a_{i+1}, \dots, a_{r} \in A$, $\lvert a_{n} \omega_{n} \rvert < \lvert T^{s} \omega_{i}/a \rvert$ for all $n$ with $i < n \leq r$ and at least one $n$ satisfies 
$\lvert a_{n}\omega_{n} \rvert \geq \lvert T^{s} \omega_{i+1}/a \rvert$. As $\lvert T^{s}/a \rvert < 1$, all the factors on the right hand side are $\geq 1$. We read off:
\begin{equation} \label{Eq.Estimate-absolute-value-of-esi-over-esi+1}
	\lvert e_{s,i} \rvert \geq \lvert e_{s,i+1} \rvert, \quad \text{with equality if and only if $\lvert \omega_{i} \rvert = \lvert \omega_{i+1} \rvert$, if and only if $\boldsymbol{\omega} \in \mathbf{F}_{i}$}.
\end{equation}
\subsection{} Now fix $i$ and consider for $0 \leq s < d-1$:
\begin{equation} \label{Eq.Absolute-value-of-quotient-es+1i-esi}
	\lvert e_{s+1,i} / e_{s,i} \rvert = \lvert T \rvert \sideset{}{'} \prod_{a_{i+1}, \dots, a_{r}} \lvert T \rvert \prod \frac{\lvert T^{s+1} \omega_{i}/a \rvert}{\lvert a_{i+1} \omega_{i+1} + \dots + a_{r}\omega_{r} \rvert},
\end{equation}
where the first product $\sideset{}{'}{\textstyle\prod}$ is over those $a_{i+1}, \dots, a_{r} \in A$ such that $\lvert a_{n} \omega_{n} \rvert < \lvert T^{s} \omega_{i}/a \rvert$ and the second product $\prod$ over those
$a_{i+1}, \dots, a_{r}$ such that $\lvert a_{n} \omega_{n} \rvert < \lvert T^{s+1} \omega_{i}/a \rvert$ but for at least one $n$ $\lvert a_{n} \omega_{n} \rvert \geq \lvert T^{s}\omega_{i}/a \rvert$ holds.
We remark that each of the factors $f$ of the second product satisfies
\begin{equation}
	1 < f \leq q \quad \text{and even $f=q$ if $\lambda(\boldsymbol{\omega}) \in \mathcal{W}(\mathds{Z})$}.
\end{equation}
Hence always
\begin{equation} \label{Eq.Permanent-estimate-absolute-value-es+1i-esi}
	\lvert e_{s+1,i} \rvert \geq q \lvert e_{s,i} \rvert.
\end{equation}
If $i=r$ then both products in \eqref{Eq.Absolute-value-of-quotient-es+1i-esi} are empty, and so
\begin{equation} \label{Eq.Absolute-values-of-es+1r-and-esr}
	\lvert e_{s+1,r} \rvert = q \lvert e_{s,r} \rvert \quad \text{and} \quad \lvert e_{s,r} \rvert = q^{s-d} \quad \text{as} \quad \lvert e_{0,r} \rvert = \lvert a \rvert^{-1} = q^{-d}.
\end{equation}

\begin{Remark}
	As a function on $\Omega$, $\boldsymbol{\omega} \mapsto e_{s,i}(\boldsymbol{\omega})$ vanishes nowhere, so its absolute value 
	$\lvert e_{s,i}(\boldsymbol{\omega}) \rvert \eqdef \lVert e_{s,i} \rVert_{\mathbf{x}}$ is constant on $\Omega_{\mathbf{x}}$, where $\mathbf{x} = \lambda(\boldsymbol{\omega})$. In fact, we see from
	\eqref{Eq.Absolute-value-of-esi}  that as long as $\mathbf{x} \in \mathcal{W}$, $\log_{q} \lVert e_{s,i} \rVert$ depends only on the $x_{i} = \log \omega_{i}$ and the dimensions of certain $\mathds{F}$-subspaces of $A$ defined by
	degree limitations that depend only on the $x_{i}$. Furthermore, the dependence on $a$ enters only via the degree $d$ of $a$.
\end{Remark}

We may resume the preceding discussion as follows.

\begin{Proposition} \label{Proposition.Absolute-values-of-the-esi}
	Let $\boldsymbol{\omega} \in \mathbf{F}$ with image $\lambda(\boldsymbol{\omega}) \eqdef \mathbf{x} \in \mathcal{W}$.
	\begin{enumerate}[label=$\mathrm{(\roman*)}$]
		\item The absolute values $\lvert e_{s,i} \rvert$ and therefore the spectrum of $_{a}\phi^{\boldsymbol{\omega}}$ depend only on $\mathbf{x}$ and $d = \deg a$.
	\end{enumerate}
	We have the following monotonicity properties of $\lvert e_{s,i} \rvert$:
	\begin{enumerate}[label=$\mathrm{(\roman*)}$] \stepcounter{enumi}%
		\item For $s$ fixed and $1 \leq i < r$, $\lvert e_{s,i} \rvert \geq \lvert e_{s,i+1} \rvert$, with equality if and only if $\lvert \omega_{i} \rvert = \lvert \omega_{i+1} \rvert$, i.e, $\mathbf{x} \in \mathcal{W}_{i}$;
		\item For $i$ fixed and $0 \leq s < d-1$, $\lvert e_{s+1,i} \rvert \geq q\lvert e_{s,i} \rvert$.
	\end{enumerate}
	If moreover $\mathbf{x} \in \mathcal{W}(\mathds{Z})$ then $\log_{q} \lVert e_{s,i} \rVert_{\mathbf{x}}$ is an integer larger or equal to ${-}d$. \hfill \mbox{$\square$}
\end{Proposition}

\subsection{}\label{Sub.Suitably-ordered-F-SMB} As in \ref{Sub.x-in-WIQ}  we define a distinguished $\mathds{F}$-SMB of $_{a}\phi^{\boldsymbol{\omega}}$ by arranging the basis $_{a}B_{\boldsymbol{\omega}}$ in a suitable order. Define
\begin{equation}
	e_{s,i} \leq e_{s',i'} \Longleftrightarrow \lvert e_{s,i} \rvert < \lvert e_{s',i'} \rvert \quad \text{or} \quad (\lvert e_{s,i} \rvert = \lvert e_{s',i'} \rvert \text{ and } i > i').
\end{equation} \stepcounter{subsubsection}%
By \ref{Proposition.Absolute-values-of-the-esi}(iii) this is a total order on $_{a}B_{\boldsymbol{\omega}}$. All these data depend only on $\mathbf{x} = \lambda(\boldsymbol{\omega})$ and $d = \deg(a)$. To make this explicit, let 
$_{d}V_{\mathbf{x}}$ be an $\mathds{F}$-vector space of dimension $rd$ with basis $_{d}B_{\mathbf{x}} = \{ _{d}e_{s,i} \mid 0 \leq s < d, 1 \leq i \leq r \}$, where the $_{d}e_{s,i}$ are formal symbols. 
We define the isomorphism of $\mathds{F}$-vector spaces
\begin{equation}
	\begin{split}
		_{a}\kappa_{\boldsymbol{\omega}} \colon _{a}\phi^{\boldsymbol{\omega}} 	&\overset{\cong}{\longrightarrow} {}_{d}V_{\mathbf{x}}, \\
																														e_{s,i}							&\longmapsto {}_{d}e_{s,i}
	\end{split}
\end{equation} \stepcounter{subsubsection}%
and transfer both the absolute value \enquote{$\lvert \cdot \rvert$} on $_{a}\phi^{\boldsymbol{\omega}}$ and the order \enquote{$\leq$} on $_{a}B_{\boldsymbol{\omega}}$ to $_{d}V_{\mathbf{x}}$ resp. to
$_{d}B_{\mathbf{x}}$ via $_{a}\kappa_{\boldsymbol{\omega}}$. Then the structure of \enquote{normed $\mathds{F}$-vector space} of  $(_{d}V_{\mathbf{x}}, \lvert \cdot \rvert)$ depends only on 
$d=\deg(a)$ and $\mathbf{x} = \lambda(\boldsymbol{\omega})$, but not on $a$ and $\boldsymbol{\omega}$ themselves. For example,
\begin{equation}
	\Big\lvert \sum c_{s,i} \,{}_{d}e_{s,i} \Big\rvert = \Big\lvert \sum c_{s,i} e_{s,i}(\boldsymbol{\omega}) \Big\rvert = \sup_{s,i} \{ \lvert e_{s,i}(\boldsymbol{\omega}) \rvert \mid c_{s,i} \neq 0 \}
\end{equation} \stepcounter{subsubsection}%
if $c_{s,i}$ are coefficients in $\mathds{F}$. Let now
\subsubsection{} $\{ \lambda_{1}, \lambda_{2}, \dots, \lambda_{rd} \}$ be the set $_{d}B_{\mathbf{x}}$ arranged in increasing order with respect to \enquote{$\leq$}. We call 
$(\lambda_{1}, \lambda_{2}, \dots, \lambda_{rd})$ the \textbf{characteristic sequence} of $_{d}B_{\mathbf{x}}$. Its pre-image in $_{a}\phi^{\boldsymbol{\omega}}$ is an $\mathds{F}$-SMB, the 
\textbf{characteristic sequence} of $_{a}\phi^{\boldsymbol{\omega}}$. We are now ready to show part (i) of Theorem \ref{Theorem.Characterisation-of-elements-of-BTIQ}.

\begin{proof}[Proof of Theorem \ref{Theorem.Characterisation-of-elements-of-BTIQ}(i)]
	The implication (A)$\Rightarrow$(B) is obvious from the spectral principle \ref{Subsub.Spectral-principle}, and (B) $\Leftrightarrow$ (C) is Proposition \ref{Proposition.Absolute-values-of-the-esi}(i). As to (B)$\Rightarrow$(A), we follow the scheme outlined in the proof of
	Theorem 4.8(i), in \cite{Gekeler-ta-1}.
	\begin{enumerate}[wide, label=(\alph*)]
		\item Assume that $_{a}\phi^{\boldsymbol{\omega}}$ is $k$-inseparable for $\boldsymbol{\omega} \in \Omega$. Without restriction, $\boldsymbol{\omega} \in \mathbf{F}$, so 
		$\mathbf{x} = \lambda(\boldsymbol{\omega}) \in \mathcal{W}(\mathds{Q})$. We will use the vanishing principle \ref{Sub.Vanishing-principle}, by which it suffices to verify that $\lvert {}_{a}\ell_{k} \rvert$ is non-constant on
		$\mathbf{F}_{\mathbf{x}} = \Omega_{\mathbf{x}}$.
		\item We write $(\lambda_{1}, \lambda_{2}, \dots, \lambda_{rd})$ for the characteristic sequence of $_{d}B_{\mathbf{x}}$ transfered back to $_{a}\phi^{\boldsymbol{\omega}}$ via $_{a}\kappa_{\boldsymbol{\omega}}$. From \eqref{Eq.Drinfeld.module-of-rank-r} and \eqref{Eq.Coefficient-forms},
		\[
			{}_{a}\ell_{k}(\boldsymbol{\omega}) = a s_{q^{k}-1} \{ \lambda^{-1} \mid 0 \neq \lambda \in _{a}\phi^{\boldsymbol{\omega}} \} = a \sum_{S} P(S),
		\]
		where $s_{n}$ denotes the $n$-th elementary symmetric function, $S$ runs through the family of $(q^{k}-1)$-subsets of $_{a}\phi^{\boldsymbol{\omega}} \smallsetminus \{0 \}$ and 
		$P(S) = (\prod_{\lambda \in S} \lambda)^{-1}$.
		\item Let $m+1$ (resp. $n$) be the least (resp. largest) subscript $j$ such that $\lvert \lambda_{j} \rvert = \lvert \lambda_{k} \rvert = \lvert \lambda_{k+1} \rvert$. Then $m < k < n$ and, writing
		$\lambda_{j} = e_{s,i}$, the indices $i$ appearing in $\lambda_{m+1}, \lambda_{m+2}, \dots, \lambda_{n}$ are all different by \eqref{Eq.Permanent-estimate-absolute-value-es+1i-esi}. Therefore, $n-m \leq r$.
		\item Some $P(S)$ has largest absolute value if $S$ contains all the $q^{m}-1$ elements of $V' \smallsetminus \{0\}$ and $q^{k} - q^{m}$ elements of $V \smallsetminus V'$,
		where (here, in conflict with our general notation) $V = \sum_{1 \leq j \leq n} \mathds{F}\lambda_{j}$ and $V' = \sum_{1 \leq j \leq m} \mathds{F}\lambda_{j}$. The contribution of such $S$ to
		$a^{-1} {}_{a}\ell_{k}$ is
		\[
			P \defeq \Big( \sideset{}{'}\prod_{\lambda \in V'} \lambda \Big)^{-1} \sum_{\substack{S' \subset V \smallsetminus V' \\ \#(S') = q^{k} - q^{m}}} P(S'), \qquad P(S') = \Big( \prod_{\lambda \in S'} \lambda \Big)^{-1}.
		\]
		All the $P(S)$ of such $S$ have the same absolute value $\lvert P(S) \rvert \eqdef c$, which depends only on $\lvert \lambda_{1} \rvert, \dots, \lvert \lambda_{n} \rvert$. Write $x \equiv y$ if $\lvert x-y \rvert <c$.
		Then 
		\[
			a^{-1} {}_{a}\ell_{k} (\boldsymbol{\omega}) \equiv P \equiv \alpha_{k}(V),
		\]
		all of which are homogeneous functions of $\boldsymbol{\omega} \in \mathbf{F}_{\mathbf{x}}$ of weight $1 - q^{k}$, that is, functions $f$ on the cone $\mathbf{F}_{\mathbf{x}}^{*}$ in $\Omega^{*}$ above
		$\mathbf{F}_{\mathbf{x}}$ that satisfy $f(t\boldsymbol{\omega}) = t^{1-q^{k}}f(\boldsymbol{\omega})$ for constants $t \in C_{\infty}^{*}$ and $\boldsymbol{\omega} \in \mathbf{F}_{\mathbf{x}}^{*}$.
		\item We must show that $\lvert \alpha_{k}(V) \rvert$ is not constant on $\mathds{F}_{\mathbf{x}}$. This can now be copied verbatim from the corresponding parts of the proof of Theorem 4.8(i) in
		\cite{Gekeler-ta-1}. Therefore, condition (B) of Theorem \ref{Theorem.Characterisation-of-elements-of-BTIQ} implies (A), and the theorem is proved. \qedhere
	\end{enumerate}
\end{proof}

For later use we note the following fact.

\begin{Corollary}[to the preceding proof] \label{Corollary.Spectral-norm-of-coefficient-forms}
	Let $(\lambda_{1}, \lambda_{2}, \dots, \lambda_{rd})$ be the characteristic sequence of $_{d}V_{\mathbf{x}}$ and $W_{k}$ the $\mathds{F}$-subspace generated by $\lambda_{1}, \dots, \lambda_{k}$.
	Then
	\begin{equation} \label{Eq.Spectral-norm-of-coefficient-forms}
		\lvert a \rvert^{-1} \lVert {}_{a}\ell_{k} \rVert_{\mathbf{x}} = \sideset{}{'}\prod_{\lambda \in W_{k}} \lvert \lambda \rvert^{-1}.
	\end{equation}
\end{Corollary}

\begin{proof}
	As in (b) we use 
	\[
		a^{-1} {}_{a}\ell_{k} = \sum_{\substack{S \subseteq {}_{a}\phi^{\boldsymbol{\omega}} \smallsetminus \{0\} \\ \#(S) = q^{k}-1}} P(S), \qquad P(S) = \Big( \prod_{\lambda \in S} \lambda \Big)^{-1}.
	\]
	Now the right hand side of \eqref{Eq.Spectral-norm-of-coefficient-forms} is the size of the terms $P(S)$ of largest value. As holomorphic functions on the affinoid $\mathbf{F}_{\mathbf{x}}$, the finitely many $P(S)$ cannot cancel, i.e.,
	$\lVert \sum_{S} P(S) \rVert_{\mathbf{x}} = \sup_{S} \lVert P(S) \rVert_{\mathbf{x}}$.
\end{proof}

Our next task is to describe $\mathcal{W}({}_{a}\ell_{k})$, that is, to find out for which $\mathbf{x} \in \mathcal{W}(\mathds{Q})$ the corresponding spaces $_{a}\phi^{\boldsymbol{\omega}}$ (or rather the spaces
$_{d}V_{\mathbf{x}}$) are $k$-inseparable. As ${}_{a}\ell_{k} = 0$ if $k > rd$, and has no zeroes on $\Omega$ if $k=rd$, we always tacitly assume that $k<rd$.

\begin{Definition}
	Put 
	\begin{multline*}
		\mathcal{W}_{d}'(k) \defeq \mathcal{W}({}_{a}\ell_{k}) = \{ \mathbf{x} \in \mathcal{W}(\mathds{Q}) \mid {}_{d}V_{\mathbf{x}} \text{ is $k$-inseparable} \} = \lambda(\mathbf{F}({}_{a}\ell_{k}) \\
			\text{and} \qquad \mathcal{BT}_{d}'(k) \defeq \Gamma \mathcal{W}_{d}'(k) = \lambda(\Omega( {}_{a}\ell_{k} ),
	\end{multline*}
	where $a \in A$ is some element of degree $d$. 
\end{Definition}

Then, as ${}_{T}\ell_{1} = g_{1} = [T^{q}-T]\alpha_{1}$, we have 
\begin{equation}
	\mathcal{W}_{1}'(1) = \mathcal{W}(1) = \mathcal{W}_{r-1}.
\end{equation}
There is a very satisfactory description of $\mathcal{W}_{d}'(k)$, and, related, of $\lVert {}_{a}\ell_{k} \rVert_{\mathbf{x}}$, provided that $k \leq d$.

\begin{Theorem} \label{Theorem.On-simpliciality-of-modular-forms}
	Suppose that $k \leq d = \deg a$. Then
	\begin{enumerate}[label=$\mathrm{(\roman*)}$]
		\item $\mathcal{BT}_{d}'(k) = \mathcal{BT}(k) = \mathcal{BT}(\alpha_{k})$;
		\item The modular form ${}_{a}\ell_{k}$ is simplicial;
		\item The van der Put transform $P({}_{a}\ell_{k})$ agrees with $P(\alpha_{k})$;
		\item $\log {}_{a}\ell_{k}(\boldsymbol{\omega})$ is constant on the subspace $\mathbf{F}_{(\omega_{r-1} > k-1)}$ of $\mathbf{F}$ with value 
		$\log {}_{a}\ell_{k}(\boldsymbol{\omega}) = \log {}_{a}\ell_{k}(A) = (d-k)q^{k} + q(q^{k}-1)/(q-1)$;
		\item For each $\mathbf{n} \in \mathcal{BT}(\mathds{Z})$, the local inner degrees $N_{\mathbf{n}}({}_{a}\ell_{k})$ and $N_{\mathbf{n}}(\alpha_{k})$ agree.
	\end{enumerate}
\end{Theorem}

Before starting with the proof, some remarks are in order.

\begin{Remarks-nn}
	\begin{enumerate}[wide, label=(\alph*)]
		\item Item (ii) of the theorem follows from (i). Hence in particular $P({}_{a}\ell_{k})$ is defined.
		\item Suppose that $f$ and $g$ are simplicial modular forms, such that $P(f)$ and $P(g)$ are defined. Then $P(f) = P(g)$ means that the functions $\mathbf{x} \mapsto \lVert f \rVert_{\mathbf{x}}$ and
		$\mathbf{x} \mapsto \lVert g \rVert_{\mathbf{x}}$ on $\mathcal{BT}(\mathds{Q})$ agree up to a multiplicative constant. Hence (iv) follows from (iii) and the corresponding statement for $\alpha_{k}$,
		see Remark \ref{Remark.Discussion-and-explanation-for-the-convenience-of-the-reader}. The formula for $\log ( {}_{a}\ell_{k}(A))$ in (iv) is standard: ${}_{a}\ell_{k}$ is the coefficient of the Drinfeld module that corresponds to the $A$-lattice $A$, which is a twist of the Carlitz
		module. The latter corresponds to the rank-1 lattice $\bar{\pi}A$ (\cite{Gekeler1988} Section 4). Now $\log {}_{a}\ell_{k}(\bar{\pi}A) = (d-k)q^{k}$ (loc. cit. (4.5)), and the term $q(q^{k}-1)/(q-1)$ takes care
		for the twist by $\bar{\pi}$.
		\item By Proposition \ref{Proposition.Relation-local-inner-degree-van-der-Put-transform}, $P(f)$ determines all the $N_{\mathbf{n}}(f)$. Therefore (v) is also a consequence of (iii).
		\item Also, once we know that $f$ is simplicial, $P(f)$ determines $\mathcal{BT}(f)$ through $\mathcal{BT}(f)(\mathds{Z}) = \{ \mathbf{n} \in \mathcal{BT}(\mathds{Z}) \mid N_{\mathbf{n}}(f) >0 \}$.
		But as long as the simpliciality of ${}_{a}\ell_{k}$ is not established, we cannot conclude (i) from (iii). This is why we first show (i).
	\end{enumerate}
\end{Remarks-nn}

\begin{proof}[Proof of Theorem \ref{Theorem.On-simpliciality-of-modular-forms}]
	It suffices to show (i) and (iii) on the fundamental domain $\mathcal{W}$, i.e., that $\mathcal{W}_{d}'(k) = \mathcal{W}(k)$ and that the transforms $P({}_{a}\ell_{k})$ and $P(\alpha_{k})$ agree on $\mathcal{W}$.
	\begin{enumerate}[label=(\roman*), wide]
		\item \begin{enumerate}[label=(\alph*)]
				\item Let $\mathbf{x} \in \mathcal{W}(\mathds{Q})$ and $\boldsymbol{\omega} \in \mathbf{F}_{\mathbf{x}} = \lambda^{-1}(\mathbf{x})$, and assume first that \fbox{$k \leq x_{r-1}$}.
				Then by the formulas \eqref{Eq.Absolute-value-of-esi} , \eqref{Eq.Absolute-values-of-es+1r-and-esr} and the definition in \ref{Sub.Suitably-ordered-F-SMB}  of $_{d}V_{\mathbf{x}}$ and its characteristic sequence $(\lambda_{1}, \lambda_{2}, \dots, \lambda_{rd})$, the $j$-th term
				$\lambda_{j}$ equals $_{d}e_{j-1,r}$ as long as $j \leq d$ and $j \leq x_{r-1}+1$. Thus for \fbox{$k < d$} $\lambda_{k} = {}_{d}e_{k-1,r}$, $\lambda_{k+1} = {}_{d}e_{k,r}$, and so 
				$\lvert \lambda_{k+1} \rvert > \lvert \lambda_{k} \rvert$. 
				
				If \fbox{$k = d \leq x_{r-1}$}, $\lambda_{k} = {}_{d}e_{k-1,r}$, $\lambda_{k+1} = {}_{d}e_{0,r-1}$, $\log \lambda_{k} = k-1 - d = -1$, $\log \lambda_{k+1} \geq x_{r-1} - d \geq 0$, so 
				$\lvert \lambda_{k+1} \rvert > \lvert \lambda_{k} \rvert$, too. Hence $\mathbf{x} \notin W_{d}'(k)$ for
				$k \leq x_{r-1}$ and $\mathbf{x} \notin \mathcal{W}(k)$ by \ref{Remark.Discussion-and-explanation-for-the-convenience-of-the-reader}.
				\item We may therefore assume that $k > x_{r-1}$, i.e., \fbox{$x_{r-1} < k \leq d$}. Let $i$ be an index with $1 \leq i < r$ and such that $x_{i} = \log \omega_{i} < k$. Then the condition on $a_{n} \in A$:
				\[
					\deg a_{n} + x_{n} < x_{i} - d \qquad \text{for } i < n \leq r
				\]
				can be achieved only for $a_{n} = 0$, $\deg a_{n} = {-}\infty$, as $x_{n} \geq 0$ and $x_{i} - d \leq x_{i} - k < 0$. Therefore the product $\sideset{}{'}{\textstyle\prod}$ in \eqref{Eq.Absolute-value-of-esi}  for the calculation of
				$\lvert e_{0,i} \rvert$ is empty, and so
				\begin{equation} \label{Eq.log-e0i-and-esi}
					\log e_{0,i} = x_{i} -d. \quad \text{Likewise,} \quad \log e_{s,i} = x_{i} + s-d,
				\end{equation} \stepcounter{subsubsection}%
				as long as the right hand side is $\leq 0$.
				\item Let $V = K_{\infty}^{r}$ with standard basis $\{e_{1}, \dots, e_{r}\}$ and the norm $\nu_{\mathbf{x}}$ be as in \ref{Sub.x-in-WIQ} . We consider 
				\[
					_{d} V \defeq \Big\{ \sum_{1 \leq i \leq r} a_{i}e_{i} \Bigm| a_{i} \in A, \deg a_{i} < d \text{ for all $i$} \Big\},
				\]
				an $\mathds{F}$-space of dimension $rd$. Under 
				\begin{equation}
					\varepsilon_{d} \colon T^{s}\omega_{i} \longmapsto {}_{d}e_{s,i}
				\end{equation} \stepcounter{subsubsection}%
				it maps isomorphically onto $_{d}V_{\mathbf{x}}$. The elements $\mathbf{v}$ of $_{d}V$ with $\log_{q}(\nu_{\mathbf{x}}(\mathbf{v}) \leq d$ are the $\mathds{F}$-linear combinations of the $T^{s}e_{i}$
				with $s+x_{i} \leq d$, and for such $\mathbf{v}$,
				\begin{equation} \label{Eq.log-epsilond}
					\log \varepsilon_{d}(\mathbf{v}) = \log_{q} ( \nu_{\mathbf{x}}(\mathbf{v}) ) - d,
				\end{equation} \stepcounter{subsubsection}%
				as follows from \eqref{Eq.log-e0i-and-esi} and the similar formula $\log e_{s,r} = \log(T^{s}\omega_{r}) - d = s-d$.
				\item Let $\lambda_{1}^{*}, \lambda_{2}^{*}, \dots$ be the characteristic sequence of $\mathbf{x}$ in $V$ as defined in \ref{Sub.x-in-WIQ}  and $(\lambda_{1}, \lambda_{2}, \dots, \lambda_{rd})$ that of
				${}_{d}V_{\mathbf{x}}$. By \eqref{Eq.log-epsilond},  $\lvert \varepsilon_{d}(\mathbf{v}) \rvert = q^{-d} \nu_{\mathbf{x}}(\mathbf{v})$ if $\mathbf{v} \in {}_{d}V$ satisfies $\nu_{\mathbf{x}}(\mathbf{v}) \leq q^{d}$.
				Thus at least for $k \leq d$, $\varepsilon_{d}(\lambda_{k}^{*}) = \lambda_{k}$ holds. In conjunction with \eqref{Eq.log-epsilond} this implies 
				\begin{equation} \label{Eq.Characterisation-of-equality-of-absolute-value-of-lambdaks}
					\lvert \lambda_{k} \rvert = \lvert \lambda_{k+1} \rvert \Longleftrightarrow \nu_{\mathbf{x}}(\lambda_{k}^{*}) = \nu_{\mathbf{x}}(\lambda_{k+1}^{*})
				\end{equation} \stepcounter{subsubsection}%
				for $k < d$. Hence suppose \fbox{$x_{r-1} < k = d$}. If $\nu_{\mathbf{x}}(\lambda_{k+1}^{*}) \leq q^{d}$, then $\varepsilon_{d}(\lambda_{k+1}^{*}) = \lambda_{k+1}$ and by \eqref{Eq.log-epsilond}, \eqref{Eq.Characterisation-of-equality-of-absolute-value-of-lambdaks} still holds.
				If however $\nu_{\mathbf{x}}(\lambda_{k+1}^{*}) > q^{d}$ then $\lambda_{k+1}^{*} = \omega_{r-1}$ (since all the $d$ norms $\nu_{\mathbf{x}}(T^{s}\omega_{r}) = q^{s}$, $0 \leq s < k = d$, are less 
				than $q^{d}$), so $q^{x_{r}-1} = \nu_{\mathbf{x}}(\omega_{r-1}) = \nu_{\mathbf{x}}(\lambda_{k+1}^{*}) > q^{d}$, in conflict with $x_{r-1} < k = d$. Hence this case cannot occur and \eqref{Eq.Characterisation-of-equality-of-absolute-value-of-lambdaks} holds for 
				all $k \leq d$. That is, $\mathcal{W}_{r}'(k) = \mathcal{W}(k)$ for $k \leq d$, and thus (i).
		\end{enumerate}\stepcounter{enumi}%
		\item \begin{enumerate}[label=(\alph*)]
			\item Consider the commutative diagram
			\[
				\begin{tikzcd}
					{}_{d}V \ar[r, hook] \ar[d, "\varepsilon_{d}"']	& {}_{d+1}V \ar[d, "\varepsilon_{d+1}"] \\
					{}_{d}V_{\mathbf{x}} \ar[r, hook, "i_{d}"']		& {}_{d+1}V_{\mathbf{x}},
				\end{tikzcd}
			\]
			where the lower inclusion $i_{d}$ is ${}_{d}e_{s,i} \mapsto {}_{d+1}e_{s,i}$. Let $_{d}\lvert \, \cdot \, \rvert$ resp. $_{d+1}\lvert \, \cdot \, \rvert$ be the absolute values on $_{d}V_{\mathbf{x}}$ 
			resp. $_{d+1}V_{\mathbf{x}}$, and $(_{d}\lambda_{k})$, $(_{d+1}\lambda_{k})$ the respective characteristic sequences. For $k \leq d$, $\varepsilon_{d}(\lambda_{k}^{*}) = {}_{d}\lambda_{k}$ 
			and $i_{d} \circ \varepsilon_{d}(\lambda_{k}^{*}) = \varepsilon_{d+1}(\lambda_{k}^{*}) = {}_{d+1}\lambda_{k}$. Hence the first $d$ terms of the characteristic sequences on $_{d}V_{\mathbf{x}}$ 
			and $_{d+1}V_{\mathbf{x}}$ agree, but $_{d}\lvert _{d}\lambda_{k} \rvert = q {}_{d+1} \lvert _{d+1} \lambda_{k} \rvert$.
			\item Let $W$ be the subspace of $_{d}V_\mathbf{x}$ generated by $\{{}_{d}\lambda_{k} \mid k \leq d \}$ and similarly 
			$W' \defeq \sum_{1 \leq k \leq d} \mathds{F}_{d+1} \lambda_{k} \subset {}_{d+1}V_{\mathbf{x}}$. Then $i_{d} \colon W \overset{\cong}{\to} W'$ and
			\begin{equation}
				_{d+1} \lvert i_{d}(\lambda) \rvert = q^{-1} {}_{d} \lvert \lambda \rvert \quad \text{for } \lambda \in W.
			\end{equation} \stepcounter{subsubsection}%
			Together with \eqref{Eq.Spectral-norm-of-coefficient-forms} it implies:
			\subsubsection{} \stepcounter{equation}%
			If $a$ of degree $d$ is replaced with $a'$ of degree $d+1$, then $\lVert {}_{a}\ell_{k} \rVert_{\mathbf{x}}$ grows by a factor $q^{(q^{k})}$, provided that $k \leq d$.
			Hence $\log_{q} \lVert {}_{a^{\prime}} \ell_{k} \rVert_{\mathbf{x}} = \log_{q} \lVert {}_{a}\ell_{k} \rVert + q^{k}$, thus $P({}_{a^{\prime}} \ell_{k}) = P({}_{a}\ell_{k})$. 
			
			We conclude that 
			\subsubsection{} \label{Subsub.Independence-of-van-der-Put-transforms}\stepcounter{equation}%
			$P({}_{a}\ell_{k})$ is independent of $a$ as long as $d = \deg a \geq k$.
			\item It is known that after a suitable scaling:
			\[
				{}_a \tilde{\ell}_{k} \defeq {}_{a}\ell_{k}/{}_{a}\ell_{k}(A), \qquad \tilde{\alpha}_{k} \defeq \alpha_{k}/\alpha_{k}(A),
			\]
			we have
			\begin{equation} \label{Eq-Lim-coefficient-forms}
				\lim_{\deg a \to \infty} {}_{a} \tilde{\ell}_{k} = \tilde{\alpha}_{k},
			\end{equation}
			where the limit is locally uniformly (\cite{Gekeler-ta-1} Theorem 4.13). Here locally uniform convergence means uniform convergence on the parts of an admissible covering of $\Omega$. Note
			that scaling $f \leadsto \tilde{f}$ doesn't change $P(f)$. 
			
			Now (iii) is a formal consequence of \ref{Subsub.Independence-of-van-der-Put-transforms} and \eqref{Eq-Lim-coefficient-forms}: Given a vertex $\mathbf{n}$ of $\mathcal{BT}$, let $\st(\mathbf{u})$ be the
			star of $\mathbf{n}$ (the full subcomplex of $\mathcal{BT}$ with vertices $\mathbf{n}$ and the neighbors of $\mathbf{n}$). Let $d = \deg a$ be large enough so that ${}_{a} \tilde{\ell}_{k}$ is close to 
			$\tilde{\alpha}_{k}$ on $\lambda^{-1}(\st(\mathbf{n}))$. Then 
			\[
				P({}_{a}\ell_{k}) = P( {}_{a} \tilde{\ell}_{k}) =_{\mathbf{n}} P(\tilde{\alpha}_{k}) = P(\alpha_{k}),
			\]
			where \enquote{$=_{\mathbf{n}}$} means that both sides agree on arrows belonging to $\st(\mathbf{n})$. Since this holds for all $\mathbf{n} \in \mathcal{BT}(\mathds{Z})$, and taking \ref{Subsub.Independence-of-van-der-Put-transforms} into
			account, (iii) results.
		\end{enumerate}
	\end{enumerate}
\end{proof}

The behavior of ${}_{a}\ell_{k}$ and $\mathcal{W}({}_{a}\ell_{k})$ for $k > d$ seems to be more complex, as some experiments show. It certainly deserves a deeper investigation. Here we restrict to 
present the following simple example.

\begin{Example}
	Let $a \in A$ have degree $d \geq 1$ and $k = rd-1$. Then $\mathcal{W}({}_{a}\ell_{k}) = \mathcal{W}(g_{r-1}) = \mathcal{W}_{1}$.
\end{Example}

\begin{proof}
	We must show that for $\mathbf{x} \in \mathcal{W}(\mathds{Q})$ and $\boldsymbol{\omega} \in \mathbf{F}_{\mathbf{x}}$, the lattice ${}_{a}\phi^{\boldsymbol{\omega}}$ is $k$-inseparable if and only if 
	$x_{1} = x_{2}$, where $x_{i} = \log \omega_{i}$ $(i=1,2)$. But this is obvious from \eqref{Eq.Estimate-absolute-value-of-esi-over-esi+1}  and \eqref{Eq.Absolute-values-of-es+1r-and-esr}, as $\lambda_{rd} = e_{d-1,1}$ and $\lambda_{k} = \lambda_{rd-1}$ is either $e_{d-2,1}$ (which is strictly 
	smaller than $e_{d-1,1}$) or $e_{d-1,2}$.
\end{proof}

We conclude this section with some hints how to calculate the spectral norms of ${}_{a}\ell_{k}$.

\subsection{} Let $\mathbf{x} \in \mathcal{W}(\mathds{Q})$, $\boldsymbol{\omega} \in \mathbf{F}_{\mathbf{x}}$, and $(\lambda_{1}, \dots, \lambda_{rd})$ be the characteristic sequence. The numbers
$m_{s} \defeq \log \lambda_{s}$ are calculated in \ref{Sub.Product-representation-of-esi} . Let $(W_{s})_{1 \leq s \leq rd}$ be the associated filtration on $_{a}\phi^{\boldsymbol{\omega}}$, $W_{s} \defeq \sum_{1 \leq t \leq s} \mathds{F}\lambda_{t}$.
Each element $\lambda$ of $W_{s} \setminus W_{s-1}$ satisfies
\begin{equation}
	\log \lambda = m_{s},
\end{equation}
and Corollary \ref{Corollary.Spectral-norm-of-coefficient-forms}  implies that
\begin{equation}
	\log_{q} \lVert {}_{a}\ell_{k} \rVert_{\mathbf{x}} = d - \sideset{}{'}\sum_{\lambda \in W_{k}} \log \lambda.
\end{equation}
In the most simple case, where $\mathbf{x}$ is the origin $\mathbf{0}$, this yields the following.

\begin{Proposition} \label{Proposition.Spectral-norm-of-kth-coefficient-of-operator-polynomial}
	Let $a \in A$ have degree $d \geq 1$ and ${}_{a}\ell_{k}$ be the $k$-th coefficient of the operator polynomial
	\[
		\phi_{a}^{\boldsymbol{\omega}}(X) = \sum_{0 \leq k \leq rd} {}_{a}\ell_{k}(\boldsymbol{\omega}) X^{q^{k}}.
	\]
	Write $k = k_{0} + sr$ with $0 \leq k_{0} < r$, $0 \leq s \leq d$. Then the spectral norm $\lVert {}_{a}\ell_{k} \rVert_{\mathbf{0}}$ on $\mathbf{F}_{\mathbf{0}}$ is given by 
	\[
		\log_{q} \lVert {}_{a}\ell_{k} \rVert_{\mathbf{0}} = (d-s)q^{k} + q^{r}(q^{rs} - 1)/(q^{r}-1).
	\]
\end{Proposition}

\begin{proof}
	With the terminology of \ref{Sub.Product-representation-of-esi} , the characteristic sequence of $_{a}\phi^{\boldsymbol{\omega}}$ is 
	\[
		e_{0,r}, e_{0,r-1}, \dots, e_{0,1}, e_{1,r} \dots, e_{1,1}, \dots, e_{d-1,r}, \dots, e_{d-1,1}
	\] 
	with $\log e_{s,i} = s-d$. The result now follows from an elementary calculation left to the reader.
\end{proof}

\begin{Corollary}
	For $1 \leq k \leq rd$ the origin $\mathbf{0}$ belongs to $\mathcal{W}_{d}'(k) = \lambda(\mathbf{F}({}_{a}\ell_{k}))$ if and only if $k \not\equiv 0 \pmod{r}$.
\end{Corollary}

\begin{proof}
	This follows from the description of the characteristic sequence in the proof of \ref{Proposition.Spectral-norm-of-kth-coefficient-of-operator-polynomial}.
\end{proof}

\section{A case study: The para-Eisenstein series $\alpha_{2}$ in rank $r=3$}

In this section the rank $r$ always equals 3. We work out in detail the behavior of the $p$-Eisenstein series $\alpha_{2}$ above points $\mathbf{x} \in \mathcal{W}(\mathds{Z}) = \mathcal{W}^{3}(\mathds{Z})$.

\subsection{} Consider $\mathbf{x} \in \mathcal{W}(\mathds{Q})$, $\mathbf{x} = (x_{1}, x_{2}, x_{3}) = (x_{1}, x_{2}, 0)$ with $x_{1} \geq x_{2} \geq 0$ and $\boldsymbol{\omega} \in \mathbf{F}_{\mathbf{x}}$. If
$x_{2} > 1$ the characteristic sequence of $\Lambda_{\boldsymbol{\omega}}$ is $\lambda_{1} = 1$, $\lambda_{2} = T$ and $\lvert \lambda_{3} \rvert > q$. According to our general recipe (see \ref{Remark.Discussion-and-explanation-for-the-convenience-of-the-reader} and
\ref{Procedure.To-determine-spectral-norm-of-alpha-k}):
\begin{equation}
	\alpha_{2}(\boldsymbol{\omega}) \equiv s_{q^{2}-1} \{ \lambda^{-1} \mid 0 \neq \lambda \in \mathds{F} + \mathds{F}T \} \equiv \alpha_{2}(A).
\end{equation} \stepcounter{subsubsection}%
Here \enquote{$\equiv$} means (see proof of Theorem \ref{Theorem.Characterisation-of-elements-of-BTIQ}), part (d)) equality of the leading terms. Hence $\lVert \alpha_{2} \rVert_{\mathbf{x}}$ is constant on $\{ \mathbf{x} \mid x_{2} > 1 \}$ and
by continuity, on $\{ \mathbf{x} \mid x_{2} \geq 1 \} \eqdef \mathcal{W}_{(x_{2} \geq 1)}$ with value
\begin{equation} \label{Eq.Log-of-spectral-norm-of-alpha-2}
	\log_{q} \lVert \alpha_{2} \rVert_{\mathbf{x}} = -(q^{2} - q).
\end{equation} \stepcounter{subsubsection}%
A fortiori, the van der Put transform $P(\alpha_{2})$ is identically zero on $W_{(x_{2} \geq 1)}$.

Let now $\mathbf{x}$ be $(x_{1}, 1,0)$ with $x_{1} \geq 1$. Then $\lambda_{1} = 1$, $\lambda_{2} = T$, $\lambda_{3} = \omega_{2}$ and 
\[
	\lambda_{4} = \begin{cases} \omega_{1}, ~\lambda_{5} = T^{2},	&\text{if } x_{1} = 1, \\ T^{2},	&\text{if } x_{1} > 1. \end{cases}
\]
Hence
\begin{equation}
	\alpha_{2}(\boldsymbol{\omega}) \equiv s_{q^{2}-1} \{ \lambda^{-1} \mid 0 \neq \lambda \in W \},
\end{equation} \stepcounter{subsubsection}%
with
\[
	W = \begin{cases} \mathds{F} + \mathds{F}T + \mathds{F}\omega_{2} + \mathds{F}\omega_{1},	&\text{if } x_{1} = 1, \\ \mathds{F} + \mathds{F}T + \mathds{F}\omega_{2},		&\text{if } x_{1} > 1. \end{cases}
\]
Finally, consider $\mathbf{x} = (x_{1}, 0,0)$ with $x_{1} \geq 0$. Then $\lambda_{1} = 1$, $\lambda_{2} = \omega_{2}$, 
\[
	\lambda_{3} = \begin{cases} \omega_{1}, ~\lambda_{4} = T,	 &\text{if } x_{1} = 0, \\ T,	&\text{if } x_{1} > 0. \end{cases}
\]
Hence
\begin{equation}
	\alpha_{2}(\boldsymbol{\omega}) \equiv s_{q^{2} - 1} \{ \lambda^{-1} \mid 0 \neq \lambda \in W \},
\end{equation} \stepcounter{subsubsection}%
where now
\[
	W = \begin{cases} \mathds{F} + \mathds{F}\omega_{2} + \mathds{F}\omega_{1},	&\text{if } x_{1} = 0, \\ \mathds{F} + \mathds{F}\omega_{2},	&\text{if } x_{1} > 0 \end{cases}
\]
and $\log_{q} \lVert \alpha_{2} \rVert_{\mathbf{x}} = 0$ in both cases by \ref{Procedure.To-determine-spectral-norm-of-alpha-k}.

\begin{Table} \label{Table.Weyl-chamber}
	The Weyl chamber $\mathcal{W} = \mathcal{W}^{3}$. Vertices $\mathbf{n} = (n_{1}, n_{2}, 0)$ are designated by $(n_{1}, n_{2})$.
	
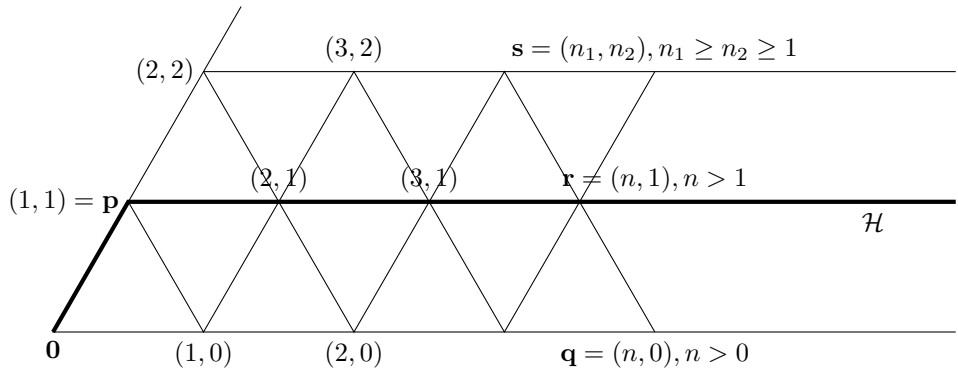
\begin{center}
		\begin{tikzpicture}[scale=2]
	    	\draw (6,0) -- (0,0) -- (1.25, {1.25*sqrt(3)});
	    	\draw (0.5, {sqrt(3)/2}) -- (1,0) -- (2, {sqrt(3)}) -- (3,0) -- (4, {sqrt(3)});
	    	\draw (1, {sqrt(3)}) -- (2,0) -- (3, {sqrt(3)}) -- (4,0);
	  		\draw (1, {sqrt(3)}) -- (6, {sqrt(3)});
	    	\draw[ultra thick] (0,0) -- (0.5,{sqrt(3)/2}) -- (6, {sqrt(3)/2}) node[below, pos=0.9] {$\mathcal{H}$};
	    	\node[below] (0) at (0,0) {$\mathbf{0}$};
	    	\node[below] (10) at (1,0) {$(1,0)$};
	    	\node[below] (20) at (2,0) {$(2,0)$};
	    	\node[below] (q) at (4,0) {$\mathbf{q} = (n,0), n> 0$};
	    	\node[left] (p) at (0.5, {sqrt(3)/2}) {$(1,1) = \mathbf{p}$};
	    	\node[left] (22) at (1,{sqrt(3)}) {$(2,2)$};
	    	\node[above] (32) at (2,{sqrt(3)}) {$(3,2)$};
	    	\node[above] (21) at (1.5, {sqrt(3)/2}) {$(2,1)$};
	    	\node[above] (31) at (2.5, {sqrt(3)/2}) {$(3,1)$};
	    	\node[above] (s) at (4,{sqrt(3)}) {$\mathbf{s} = (n_{1}, n_{2}), n_{1} \geq n_{2} \geq 1$};
	    	\node[above] (r) at (4,{sqrt(3)/2}) {$\mathbf{r} = (n,1), n>1$};
	\end{tikzpicture}
	\end{center}
	By (\ref{Theorem.Characterisation-of-certain-subsets-of-BTIQ}), $\mathcal{W}(2) = \lambda(\mathbf{F}(\alpha_{2}))$ is the union of the half-line $\mathcal{H}$ and the edge $(\mathbf{0}, \mathbf{p})$. We conclude from the above that 
	\begin{itemize}
		\item $P(\alpha_{2})$ vanishes identically on the domain above the horizontal half-line $\mathcal{H} = ( (1,1), (2,1), (3,1), \dots )$ and on $\mathcal{H}$;
		\item $P(\alpha_{2})$ vanishes on arrows that belong to $\mathcal{W}_{2} = ( (0,0), (1,0), (2,0), \dots )$;
		\item $P(\alpha_{2})(e) = -(q^{2}-q)$ (resp. $q^{2}-q$) for arrows $e$ that connect $\mathcal{W}_{2}$ to $\mathcal{H}$ (resp. $\mathcal{H}$ to $\mathcal{W}_{2}$).
	\end{itemize}
\end{Table}

We also note that
\subsubsection{} arrows (= oriented 1-simplices) of $\mathcal{W}$ are of type 1 (as defined in \eqref{Eq.Definition-van-der-Put-transform}), if horizontal and oriented east to west or sloped and oriented southwest to northeast or northwest to southeast.

\subsection{} From now on we assume that $\mathbf{n} = (n_{1},n_{2}) = (n_{1}, n_{2}, 0) \in \mathcal{W}(\mathds{Z})$. According to the cases 
\begin{enumerate}[label=(\alph*)]
	\item $\mathbf{n} = \mathbf{0}$,
	\item $\mathbf{n} = (n,0)$, $n > 0$,
	\item $\mathbf{n} = (n,n,0)$, $n > 0$,
	\item $(n_{1}, n_{2})$, $n_{1} > n_{2} > 0$,
\end{enumerate}
the fixed group $\Gamma_{\mathbf{n}}$ in $\Gamma = \GL(3, A)$ is
\subsubsection{}\label{Subsub.Fixed-group-Gamma-n} \stepcounter{equation}%
\begin{enumerate}[label=(\alph*)]
	\item $\GL(3, \mathds{F}$);
	\item $\left\{ \begin{tikzpicture}[scale=0.4, baseline={([yshift=-\the\fontdimen22\textfont2]current bounding box.center)}]
	\draw (0,0) rectangle (3,3);
	\draw (1,0) -- (1,3);
	\draw (0,2) -- (3,2);
	\draw (0,1) -- (1,1);
	\draw (2,3) -- (2,2);
	
	\node (1) at (0.5,2.5) {$*$};
	\node (2) at (1.5, 2.5) {$b$};
	\node (3) at (2.5, 2.5) {$c$};
	\node (4) at (0.5, 1.5) {$0$};
	\node (5) at (0.5, 0.5) {$0$};
	\node (6) at (2,1) {$*$};
\end{tikzpicture} \right\}$, where the $*$ stands for an element of $\mathds{F}^{*}$ or of $\GL(2,\mathds{F})$, respectively, and $b,c \in A$ such that $\deg b, \deg c \leq n$;
	\item $\left\{ \begin{tikzpicture}[scale=0.4, baseline={([yshift=-\the\fontdimen22\textfont2]current bounding box.center)}]
	\draw (0,0) rectangle (3,3);
	\draw (1,0) -- (1,1);
	\draw (2,0) -- (2,3);
	\draw (0,1) -- (3,1);
	\draw (2,2) -- (3,2);
	
	\node(1) at (1,2) {$*$};
	\node (2) at (2.5,2.5) {$c$};
	\node (3) at (2.5,1.5) {$f$};
	\node (4) at (2.5,0.5) {$*$};
	\node (5) at (0.5,0.5) {$0$};
	\node (6) at (1.5, 0.5) {$0$};
\end{tikzpicture} \right\}$, $c,f \in A$ with $\deg c, \deg f \leq n$;
	\item $\left\{ \begin{tikzpicture}[scale=0.4, baseline={([yshift=-\the\fontdimen22\textfont2]current bounding box.center)}]
	\draw (0,0) rectangle (3,3);
	\draw (0,2) -- (3,2);
	\draw (1,0) -- (1,3);
	\draw (0,1) -- (3,1);
	\draw (2,0) -- (2,3);
	
	\node (1) at (0.5, 2.5) {$*$};
	\node (2) at (1.5, 2.5) {$b$};
	\node (3) at (2.5, 2.5) {$c$};
	\node (4) at (0.5, 1.5) {$0$};
	\node (5) at (1.5, 1.5) {$*$};
	\node (6) at (2.5,1.5) {$f$};
	\node (7) at (0.5, 0.5) {$0$};
	\node (8) at (1.5, 0.5) {$0$};
	\node (9) at (2.5, 0.5) {$*$};
\end{tikzpicture} \right\}$, $b,c,f \in A$, $\deg b \leq n_{1} - n_{2}$, $\deg c \leq n_{1} - n_{3} = n_{1}$, $\deg f \leq n_{2} - n_{3} = n_{2}$.
\end{enumerate}
( $\GL(3, K_{\infty})$ and its subgroup $\Gamma$ acts as a matrix group from the right on $V = K_{\infty}^{3} = \{ \text{row vectors} \}$ and therefore on the set of classes $[L]$ of $O_{\infty}$-lattices $L$ in $V$.
This induces the familiar left action of $\GL(3, K_{\infty})$ on $\Omega = \Omega^{3}$, see \cite{Gekeler-ta-2} Section (1.3), (1.5) for more details.)

\subsubsection{} \stepcounter{equation}%
Let $\overline{(\,\cdot \,)}$ be the map from $\Gamma_{\mathbf{n}}$ to $\GL(3, \mathds{F})$ which maps each entry to its leading coefficient and 
$\bar{\Gamma}_{\mathbf{n}} \hookrightarrow \GL(3,\mathds{F})$ its image. Hence $\bar{\Gamma}_{\mathbf{n}} = \GL(3, \mathds{F})$ in case (a), the maximal parabolic subgroups in cases (b), (c), and the standard 
Borel subgroup in case (d). Then $\bar{\Gamma}_{\mathbf{n}}$ acts on the star $\st(\mathbf{n})$ of $\mathbf{n}$ in $\mathcal{BT}^{3}$, that is, on the vertices and simplices contiguous with $\mathbf{n}$.
Let $L_{\mathbf{n}} = (\pi^{n_{1}} O_{\infty}, \pi^{n_{2}} O_{\infty}, O_{\infty}) \subset V$ be the $O_{\infty}$-lattice that corresponds to $\mathbf{n}$. Then
\begin{equation} \label{Eq.Isomorphism-for-projective-spaces}
	\mathds{P}(L_{\mathbf{n}}/ \pi L_{\mathbf{n}} ) \overset{\cong}{\longrightarrow} \mathds{P}^{2}(\mathds{F})
\end{equation}
corresponds naturally to the set $\mathbf{A}_{\mathbf{n},1}$ of arrows of type 1 emanating from $\mathbf{n}$. We will determine $P(\alpha_{2})(e)$ for all $e \in \mathbf{A}_{\mathbf{n},1}$ (which requires
determining the orbits of $\bar{\Gamma}_{\mathbf{n}}$ on $\mathds{P}(L_{\mathbf{n}}/ \pi L_{\mathbf{n}} )$ and some factors of automorphy) and the local inner degree $N_{\mathbf{n}}(\alpha_{2})$
defined in \ref{Sub.Logarithm-of-invertible-holomorphic-function}.

\subsection{} Given a non-zero vector $\mathbf{y} \in V$, the \textbf{shift} of $\mathbf{n}$ toward $\mathbf{y}$ is the vertex $w$ of $\mathcal{BT}^{3}$ corresponding to 
$(L_{\mathbf{n}} \cap K_{\infty} \mathbf{y}) + \pi L_{\mathbf{n}}$. We say that the arrow $e = (\mathbf{n}, w)$ of $\mathcal{BT}^{3}$ \textbf{points} to $\mathbf{y}$. Note that $e$ is of type 1.
\subsection{} The factor of automorphy of $\gamma \in \Gamma$ at $\boldsymbol{\omega}$ is 
\begin{equation} \label{Eq.Factor-of-automorphy-of-gamma}
	\aut(\gamma, \boldsymbol{\omega}) = \gamma_{3,1} \omega_{1} + \gamma_{3,2} \omega_{2} + \gamma_{3,3} \omega_{r}.
\end{equation}
Let $\mathbf{y}$ be the row vector $(\gamma_{3,1}, \gamma_{3,2}, \gamma_{3,3})$ of $\gamma$ and $\mathbf{z} = (0,0,1)$. In the terminology of \cite{Gekeler-ta-3} Sect. 2, $\aut(\gamma, \cdot)$ is the function
$\ell_{H,H'} = \ell_{H}/\ell_{H'}$ on $\Omega$, where $H$ and $H'$ are the hyperplanes in $V$ orthogonal with $\mathbf{y}$ and $\mathbf{z}$, respectively, under the bilinear form 
$\langle \mathbf{v}, \mathbf{v}' \rangle = \sum_{1 \leq i \leq 3} v_{i} v_{i}'$. The van der Put transform $P(\aut(\gamma, \cdot))$ is given on $\mathbf{A}_{\mathbf{n},1}$ by 
\begin{equation} \label{Eq.Van-der-Put-Transform-of-factor-of-automorphy}
	P( \aut(\gamma, \cdot))(e) = \begin{cases} {-}1,	&\text{if $e$ points to $\mathbf{y}$, not to $\mathbf{z}$}, \\ {+}1,	&\text{if $e$ points to $\mathbf{z}$, not to $\mathbf{y}$}, \\ \hphantom{+}0,	&\text{otherwise} \end{cases}
\end{equation}
(\cite{Gekeler-ta-3} 2.10). We thus get for a simplicial modular form $f$ of weight $k$, $\gamma \in \Gamma$ and an arrow $e \in \mathbf{A}_{\mathbf{n}, 1}$
\begin{equation} \label{Eq.Van-der-Put-of-simplicial-modular-form}
	P(f)(\gamma e) = kP(\aut(\gamma, \cdot))(e) + P(f)(e).
\end{equation}
We will apply this to $\alpha_{2}$ and $\gamma \in \Gamma_{\mathbf{n}}$.
\subsection{}\label{Sub.Case-distinction-o-p-q-r-s} We shall determine the relevant local data of $\alpha_{2}$ above $\mathbf{n} = (n_{1}, n_{2}) \in W(\mathds{Z})$, distinguishing the five cases
\begin{enumerate}[label=(\alph*)] \setcounter{enumi}{14}
	\item $\mathbf{n} = \mathbf{0} = (0,0)$,
	\item $\mathbf{n} = \mathbf{p} = (1,1)$,
	\item $\mathbf{n} = \mathbf{q} = (n,0)$, $n>0$,
	\item $\mathbf{n} = \mathbf{r} = (n,1)$, $n>1$,
	\item $\mathbf{n} = \mathbf{s} = (n_{1}, n_{2})$, $n_{1} \geq n_{2} > 1$.
\end{enumerate}
\subsection{} \textbf{Case} (0), $\mathbf{n} = \mathbf{0}$. A look to \ref{Table.Weyl-chamber}, along with the fact that $\mathcal{W}$ is a fundamental domain for $\Gamma$ on $\mathcal{BT}$, shows that all
$e \in \mathbf{A}_{\mathbf{0}, 1}$ are $\bar{\Gamma}_{\mathbf{0}}$-equivalent with $e_{0} = (\mathbf{0}, \mathbf{p})$. Under \eqref{Eq.Isomorphism-for-projective-spaces} , $e_{0}$ corresponds to the point $(0:0:1)$ of $\mathds{P}^{2}(\mathds{F})$
(i.e., $e_{0}$ points to $\mathbf{z}$), whose fixed group in $\bar{\Gamma}_{0} = \GL(3, \mathds{F})$ is the parabolic subgroup 
\[
	\bar{\Gamma}_{\mathbf{0}, e_{0}} = \left\{ \begin{tikzpicture}[scale=0.4, baseline={([yshift=-\the\fontdimen22\textfont2]current bounding box.center)}]
	\draw (0,0) rectangle (3,3);
	\draw (0,1) -- (3,1);
	\draw (2,0) -- (2,3);
	\draw (1,0) -- (1,1);
	
	\node (1) at (1, 2) {$*$};
	\node (2) at (2.5,2.5) {$*$};
	\node (3) at (2.5,1.5) {$*$};
	\node (4) at (0.5,0.5) {$0$};
	\node (5) at (1.5,0.5) {$0$};
	\node (6) at (2.5,0.5) {$*$};
\end{tikzpicture} \right\}.
\] 
Hence for
$\gamma \in \bar{\Gamma}_{\mathbf{0}} \smallsetminus \bar{\Gamma}_{\mathbf{0}, e_{0}}$ the recipe \eqref{Eq.Van-der-Put-Transform-of-factor-of-automorphy} gives $P(\aut(\gamma, \cdot))(e_{0}) = 1$, and by \eqref{Eq.Van-der-Put-of-simplicial-modular-form} ,
\begin{equation}
	P(\alpha_{2})(\gamma e_{0}) = q^{2} -1 + (q-q^{2}) = q-1.
\end{equation}
Therefore,
\begin{equation} \label{Eq.Case-0-sum-of-van-der-Puts}
	\sum_{e \in \mathbf{A}_{\mathbf{0},1}} P(\alpha_{2})(e) = ( [\Gamma_{\mathbf{0}} : \Gamma_{\mathbf{0}, e_{0}}] - 1)(q-1) + q-q^{2} = q^{3} - q^{2}.
\end{equation}
Now $\lVert \alpha_{2} \rVert_{\mathbf{0}} = 1$, so the reduction of $\alpha_{2}$, a rational function on the reduction 
\[
	(\bar{\mathbf{F}}_{\mathbf{0}}) = \mathds{P}^{2}/ \mathds{F} \smallsetminus \bigcup H, 
\]
where $H$ runs through the hyperplanes defined over $\mathds{F}$, see \ref{Sub.Logarithm-of-invertible-holomorphic-function}, is given in the natural coordinates $(\bar{\omega}_{1}: \bar{\omega}_{2} : \bar{\omega}_{3})$ by 
$\bar{\alpha}_{2}(\bar{\boldsymbol{\omega}}) = \alpha_{2}(\Lambda_{\bar{\boldsymbol{\omega}}})$, where 
$\Lambda_{\bar{\boldsymbol{\omega}}} = \mathds{F} \bar{\omega}_{1} + \mathds{F} \bar{\omega}_{2} + \mathds{F}\bar{\omega}_{3}$ is the $\mathds{F}$-lattice spanned by the entries of 
$\bar{\boldsymbol{\omega}}$. In the following formulas, we omit for the moment the bars $(\overline{\vphantom{)}\hphantom{\,\cdot \,}})$ on the $\omega_{i}$. By \eqref{Eq.Coefficients-of-Moore-determinant},
\begin{equation}
	\alpha_{2}(\Lambda_{\bar{\boldsymbol{\omega}}}) = \frac{M^{(2)}(\omega_{1}, \omega_{2}, \omega_{3})}{M(\omega_{1}, \omega_{2}, \omega_{3})^{q}}.
\end{equation}
The Moore determinant in the numerator is 
\[
	\det \begin{pmatrix} \omega_{1} & \omega_{1}^{q} & \omega_{1}^{q^{3}} \\ \omega_{2} & \omega_{2}^{q} & \omega_{2}^{q^{3}} \\ \omega_{3} & \omega_{3}^{q} & \omega_{3}^{q^{3}} \end{pmatrix},
\]
the denominator the $q$-th power of the \enquote{same} determinant, but the $(~)^{q^{3}}$ entries replaced with $(~)^{q^{2}}$. Now the form $M(\omega_{1}, \omega_{2}, \omega_{3})$ on 
$\mathds{P}^{2}/\mathds{F}$ has simple zeroes along the rational hyperplanes $H$ of $\mathds{P}^{2}/\mathds{F}$ and no other zeroes. As is easily seen, $M^{(2)}(\omega_{1}, \omega_{2}, \omega_{3})$ has a
simple zero along $(\omega_{1} = 0)$ and thus, by symmetry, simple zeroes along all the $H$'s. Therefore the local inner degree $N_{\mathbf{0}}(\alpha_{2})$ of $\alpha_{2}$ at $\mathbf{0}$ is
\begin{align}
	N_{\mathbf{0}}(\alpha_{2}) 	&= \deg( M^{(2)}(\boldsymbol{\omega}) ) - \deg( \text{divisor of $M^{(2)}(\boldsymbol{\omega})$ along the boundary}) \\
															&= q^{3} + q + 1 - (q^{2} + q + 1) = q^{3} - q^{2}, \nonumber
\end{align}
in accordance with \eqref{Eq.Case-0-sum-of-van-der-Puts}  and Proposition \ref{Proposition.Relation-local-inner-degree-van-der-Put-transform}. 

For later use, we remark the analogous properties of $\bar{\alpha}_{1}$:
\begin{align} \label{Eq.Properties-of-alpha-1-and-Moore-determinant}
		&\bar{\alpha}_{1}(\bar{\boldsymbol{\omega}}) = \alpha_{1}(\Lambda_{\bar{\boldsymbol{\omega}}}), \qquad \alpha_{1}(\Lambda_{\bar{\boldsymbol{\omega}}}) = \frac{M^{(1)}(\omega_{1}, \omega_{2}, \omega_{3})}{M(\omega_{1}, \omega_{2}, \omega_{3})^{q}}, \\ 
		&\qquad \qquad M^{(1)}(\omega_{1}, \omega_{2}, \omega_{3}) = \det \begin{pmatrix} \omega_{1} & \omega_{1}^{q^{2}} & \omega_{1}^{q^{3}} \\ \omega_{2} & \omega_{2}^{q^{2}} & \omega_{2}^{q^{3}} \\ \omega_{3} & \omega_{3}^{q^{2}} & \omega_{3}^{q^{3}} \end{pmatrix} \nonumber
\end{align}
also has simple zeroes at all the boundary components $H$, so 
\[
	N_{\mathbf{0}}(\alpha_{1}) = q^{3} + q^{2} + 1 - (q^{2} + q + 1) = q^{3} - q.
\]
\subsection{} \textbf{Case} (p) of \ref{Sub.Case-distinction-o-p-q-r-s} , $\mathbf{n} = \mathbf{p} = (1,1)$. By \ref{Subsub.Fixed-group-Gamma-n} , 
\[
	\bar{\Gamma}_{\mathbf{p}} = \left\{ \begin{tikzpicture}[scale=0.4, baseline={([yshift=-\the\fontdimen22\textfont2]current bounding box.center)}]
	\draw (0,0) rectangle (3,3);
	\draw (0,1) -- (3,1);
	\draw (2,0) -- (2,3);
	\draw (1,0) -- (1,1);
	
	\node (1) at (1, 2) {$*$};
	\node (2) at (2.5,2.5) {$*$};
	\node (3) at (2.5,1.5) {$*$};
	\node (4) at (0.5,0.5) {$0$};
	\node (5) at (1.5,0.5) {$0$};
	\node (6) at (2.5,0.5) {$*$};
\end{tikzpicture} \right\}.
\] 
and there are two classes modulo $\bar{\Gamma}_{\mathbf{p}}$ of arrows $e \in \mathbf{A}_{\mathbf{p}, 1}$, represented by $e_{1} = ( \mathbf{p}, (2,2))$ and $e_{2} = (\mathbf{p}, (1,0) )$. Under \eqref{Eq.Isomorphism-for-projective-spaces} , 
$e_{1}$ corresponds to $(0:0:1)$ and $e_{2}$ to $(0:1:0)$. We see that $\bar{\Gamma}_{\mathbf{p}}$ fixes $e_{1}$, while the stabilizer group of $e_{2}$ is 
\[
	\bar{\Gamma}_{\mathbf{p}, e_{2}} = \left\{ \begin{tikzpicture}[scale=0.4, baseline={([yshift=-\the\fontdimen22\textfont2]current bounding box.center)}]
	\draw (0,0) rectangle (3,3);
	\draw (0,1) -- (3,1);
	\draw (1,0) -- (1,3);
	\draw (0,2) -- (3,2);
	\draw (2,0) -- (2,3);
	
	\node (1) at (0.5,2.5) {$*$};
	\node (2) at (1.5,2.5) {$*$};
	\node (3) at (2.5,2.5) {$*$};
	\node (4) at (0.5,1.5) {$0$};
	\node (5) at (1.5,1.5) {$*$};
	\node (6) at (2.5,1.5) {$0$};
	\node (7) at (0.5,0.5) {$0$};
	\node (8) at (1.5,0.5) {$0$};
	\node (9) at (2.5,0.5) {$*$};
\end{tikzpicture} \right\},
\] 
with index $(q+1)q$ in $\bar{\Gamma}_{\mathbf{p}}$. Hence the orbit of $e_{2}$ under $\bar{\Gamma}_{\mathbf{p}}$
has length $(q+1)q$, and each $e$ which is $\bar{\Gamma}_{\mathbf{p}}$-equivalent with $e_{2}$ satisfies
\begin{equation}
	P(\alpha_{2})(e) = P(\alpha_{2})(e_{2}) = q^{2} - q,
\end{equation}
since $\aut(\gamma, \boldsymbol{\omega}) = 1$ for $\gamma \in \bar{\Gamma}_{\mathbf{p}}$. We find
\begin{equation}
	\sum_{e \in \mathbf{A}_{\mathbf{p},1}} P(\alpha_{2})(e) = (q^{2} + q)(q^{2}-q) = q^{4} - q^{2}.
\end{equation}
Now let's determine $N_{\mathbf{p}}(\alpha_{2})$ independently of Proposition \ref{Proposition.Relation-local-inner-degree-van-der-Put-transform}. For $\boldsymbol{\omega} \in \mathbf{F}_{\mathbf{p}}$, the characteristic sequence of $\Lambda_{\boldsymbol{\omega}}$ is
$\lambda_{1} = 1$, $\lambda_{2} = T$, $\lambda_{3} = \omega_{2}$, $\lambda_{4} = \omega_{1}$, while $\lvert \lambda_{5} \rvert = \lvert T^{2} \rvert > \lvert \lambda_{4} \rvert$. Therefore
\begin{equation} \label{Eq.Alpha-2-of-omega}
	\alpha_{2}(\boldsymbol{\omega}) \equiv \alpha_{2}(W'),
\end{equation}
where $W'$ is the $\mathds{F}$-space generated by $\{1,T, \omega_{2}, \omega_{1}\}$. (Only basis elements $\lambda_{j}$ with $\lvert \lambda_{j} \rvert \leq \lvert \lambda_{2} \rvert$ count for the leading term!)
In order to find the reduction, we replace $W'$ with $W \defeq \pi W'$, spanned by $\{ \pi, 1, \pi \omega_{2}, \pi \omega_{1}\}$, $\pi = T^{-1}$. Let $\bar{\alpha}_{2}$ be the function $(\overline{T^{q^{2} - q} \alpha_{2}})$
(i.e., $\alpha_{2}$ scaled such that the spectral norm equals 1, and then reduced) on $\bar{\mathbf{F}}_{\mathbf{p}} \overset{\cong}{\to} \mathds{P}^{2}/\mathds{F} \smallsetminus \bigcup \, H$ in its natural
coordinates $\bar{\boldsymbol{\omega}} = ( (\overline{\pi \omega_{1}} : (\overline{\pi \omega_{2}}) : 1)$. Since $\alpha_{4}(W) = (\sideset{}{'}{\textstyle\prod}_{\lambda \in W} \lambda)^{-1}$ as a function
on $\mathbf{F}_{\mathbf{p}}$ has constant absolute value, Lemma \ref{Lemma.Easy-lemma}  (with $d=4$ and $d_{0}=1$) yields that 
\[
	\bar{\alpha}_{2}(\bar{\boldsymbol{\omega}}) = \alpha_{1}(\bar{\boldsymbol{\omega}})^{q} \alpha,
\]
where $\alpha$ is a rational function on $\mathds{P}^{2}/\mathds{F}$ invertible on its subspace $\bar{\mathbf{F}}_{\mathbf{p}}$. Hence the degree of the zero divisor of $\bar{\alpha}_{2}$ on 
$\overline{\mathbf{F}}_{\mathbf{p}}$ is $q$ times the same quantity of $\bar{\alpha}_{1}$ on $\bar{\mathbf{F}}_{\mathbf{0}}$, i.e., 
\begin{equation} \label{Eq.Case-p-Local-inner-degree}
	N_{\mathbf{p}}(\alpha_{2}) = q N_{\mathbf{0}}(\alpha_{1}) = q^{4} - q^{2}
\end{equation}
by \eqref{Eq.Properties-of-alpha-1-and-Moore-determinant} . We have made use of the following easy lemma.

\begin{Lemma} \label{Lemma.Easy-lemma}
	Let $W \subset C_{\infty}$ be a finite $\mathds{F}$-lattice of dimension $d$, $\{ \lambda_{1}, \dots, \lambda_{d} \}$ an $\mathds{F}$-SMB, $W_{0}$ the sublattice generated by 
	$\{\lambda_{1}, \dots, \lambda_{d_{0}}\}$, where $\lvert \lambda_{d_{0}} \rvert < \lvert \lambda_{d_{0}+1} \rvert = \dots = \lvert \lambda_{d} \rvert = 1$. Write $(\bar{.})$ for the reduction map
	from $O_{C_{\infty}}$ to its residue class field $\bar{\mathds{F}}$ and $e_{W}(X) = \sum_{0 \leq j \leq d} \alpha_{j}(W)X^{q^{j}}$ for the exponential function. Let $\bar{W} = \sum_{d_{0} < j \leq d} \mathds{F} \bar{\lambda}_{j}$ be the reduced lattice of dimension $d-d_{0}$ in $\bar{\mathds{F}}$. Then for each $j$, $0 \leq j \leq d$, $\lvert \alpha_{j}(W)/\alpha_{d}(W) \rvert \leq 1$ and for $d_{0} \leq j \leq d$,
	\[
		( \overline{\alpha_{j}(W)/\alpha_{d}(W)}) = ( \alpha_{j - d_{0}}(\bar{W}) / \alpha_{d-d_{0}}(\bar{W}))^{q^{d_{0}}}
	\]
	holds.
\end{Lemma}

\begin{proof}
	Let 
	\begin{align*}
		f_{W}(X) 				&\defeq \alpha_{d}(W)^{-1} e_{W}(X) = \prod_{\lambda \in W} (X-\lambda) 
	\intertext{and correspondingly}
		f_{\bar{W}}(X) 	&= \alpha_{d-d_{0}}(\bar{W}) e_{\bar{W}}(X) = \prod_{\lambda \in \bar{W}} (X-\lambda).
	\end{align*}
	Now the coefficients $\alpha_{j}(W)/\alpha_{d}(W)$ of $f_{W}$ are $\leq 1$ in absolute value and may be reduced. As $\overline{(\hphantom{\,\cdot\,})} \colon W \to \bar{W}$ is $q^{d_{0}}$ to 1,
	$( \overline{f_{W}}) = (f_{\bar{W}})^{q^{d_{0}}}$ and the assertion results from comparing coefficients.
\end{proof}

\subsection{} \textbf{Case} (q) of \ref{Sub.Case-distinction-o-p-q-r-s} , $\mathbf{n} = \mathbf{q} = (n,0)$ with $n > 0$. According to \ref{Subsub.Fixed-group-Gamma-n} , 
\[
	\bar{\Gamma}_{\mathbf{q}} = \left\{ \begin{tikzpicture}[scale=0.4, baseline={([yshift=-\the\fontdimen22\textfont2]current bounding box.center)}]
	\draw (0,0) rectangle (3,3);
	\draw (1,0) -- (1,3);
	\draw (0,2) -- (3,2);
	\draw (0,1) -- (1,1);
	\draw (2,3) -- (2,2);
	
	\node (1) at (0.5,2.5) {$*$};
	\node (2) at (1.5, 2.5) {$*$};
	\node (3) at (2.5, 2.5) {$*$};
	\node (4) at (0.5, 1.5) {$0$};
	\node (5) at (0.5, 0.5) {$0$};
	\node (6) at (2,1) {$*$};
\end{tikzpicture} \right\}.
\] 
There are two $\bar{\Gamma}_{\mathbf{q}}$-classes of 1-arrows with origin $\mathbf{q}$, represented by $e_{1} = (\mathbf{q}, (n+1,1) )$ and $e_{2} = ( \mathbf{q}, (n-1,0))$. Under \eqref{Eq.Isomorphism-for-projective-spaces} , $e_{1}$ corresponds to 
$(0:0:1)$ and $e_{2}$ to $(1:0:0)$. We see that
\[
	\bar{\Gamma}_{\mathbf{q}, e_{1}} = \left\{ \begin{tikzpicture}[scale=0.4, baseline={([yshift=-\the\fontdimen22\textfont2]current bounding box.center)}]
	\draw (0,0) rectangle (3,3);
	\draw (0,2) -- (3,2);
	\draw (1,0) -- (1,3);
	\draw (0,1) -- (3,1);
	\draw (2,0) -- (2,3);
	
	\node (1) at (0.5, 2.5) {$*$};
	\node (2) at (1.5, 2.5) {$*$};
	\node (3) at (2.5, 2.5) {$*$};
	\node (4) at (0.5, 1.5) {$0$};
	\node (5) at (1.5, 1.5) {$*$};
	\node (6) at (2.5,1.5) {$*$};
	\node (7) at (0.5, 0.5) {$0$};
	\node (8) at (1.5, 0.5) {$0$};
	\node (9) at (2.5, 0.5) {$*$};
\end{tikzpicture} \right\} \quad \text{and} \quad \bar{\Gamma}_{\mathbf{q},e_{2}} = \left\{ \begin{tikzpicture}[scale=0.4, baseline={([yshift=-\the\fontdimen22\textfont2]current bounding box.center)}]
	\draw (0,0) rectangle (3,3);
	\draw (1,0) -- (1,3);
	\draw (0,2) -- (3,2);
	\draw (0,1) -- (1,1);
	\draw (2,3) -- (2,2);
	
	\node (1) at (0.5,2.5) {$*$};
	\node (2) at (1.5, 2.5) {$0$};
	\node (3) at (2.5, 2.5) {$0$};
	\node (4) at (0.5, 1.5) {$0$};
	\node (5) at (0.5, 0.5) {$0$};
	\node (6) at (2,1) {$*$};
\end{tikzpicture} \right\},
\]
with indices $q+1$ and $q^{2}$ in $\bar{\Gamma}_{\mathbf{q}}$, respectively. Now $e_{1}$ points to $\mathbf{z}$ but for $\gamma \in \bar{\Gamma}_{\mathbf{q}} \smallsetminus \bar{\Gamma}_{\mathbf{q},e_{1}}$
not to $\mathbf{y}$ (the third row vector of $\gamma$, see \eqref{Eq.Factor-of-automorphy-of-gamma}); hence for such $\gamma$, $P(\aut(\gamma, \cdot))(e_{1}) = +1$ and
\begin{equation}
	P(\alpha_{2})(\gamma e_{1}) = (q^{2}-1) + P(\alpha_{2})(e_{1}) = q^{2}-1 + q - q^{2} = q-1.
\end{equation}
On the other hand, $e_{2}$ points neither to $\mathbf{z}$ nor (for $\gamma \in \bar{\Gamma}_{\mathbf{q}}$) to $\mathbf{y}$, therefore $P(\aut(\gamma, \cdot))(e_{2}) = 0$ and 
\begin{equation}
	P(\alpha_{2})(\gamma e_{2}) = P(\alpha_{2})(e_{2}) = 0.
\end{equation}
Together,
\begin{equation}
	\sum_{e \in \mathbf{A}_{\mathbf{q},1}} P(\alpha_{2})(e) = (q-q^{2}) + q(q-1) + q^{2} \cdot 0 = 0.
\end{equation}
This fits with $N_{\mathbf{q}}(\alpha_{2}) = 0$ which we know a priori, as $\mathbf{q}$ does not belong to $\mathcal{W}(2)$.

\subsection{} \textbf{Case} (r) of \ref{Sub.Case-distinction-o-p-q-r-s} , $\mathbf{n} = \mathbf{r} = (n,1)$ with $n > 1$. We have 
\[
	\bar{\Gamma}_{\mathbf{r}} = \left\{ \begin{tikzpicture}[scale=0.4, baseline={([yshift=-\the\fontdimen22\textfont2]current bounding box.center)}]
	\draw (0,0) rectangle (3,3);
	\draw (0,1) -- (3,1);
	\draw (1,0) -- (1,3);
	\draw (0,2) -- (3,2);
	\draw (2,0) -- (2,3);
	
	\node (1) at (0.5,2.5) {$*$};
	\node (2) at (1.5,2.5) {$*$};
	\node (3) at (2.5,2.5) {$*$};
	\node (4) at (0.5,1.5) {$0$};
	\node (5) at (1.5,1.5) {$*$};
	\node (6) at (2.5,1.5) {$*$};
	\node (7) at (0.5,0.5) {$0$};
	\node (8) at (1.5,0.5) {$0$};
	\node (9) at (2.5,0.5) {$*$};
\end{tikzpicture} \right\},
\] 
and there are three orbits of 1-arrows emanating from $\mathbf{r}$, represented
by 
\[
	e_{1} = (\mathbf{r}, (n+1,2)) = (1:0:0), \quad e_{2} = (\mathbf{r}, (n,0)) = (0:1:0), \quad e_{3} = (\mathbf{r}, (n-1,1)) = (0:0:1),
\]
where the right hand description refers to \eqref{Eq.Isomorphism-for-projective-spaces} . The stabilizers are 
\[
	\bar{\Gamma}_{\mathbf{r}, e_{i}} = \left\{ \left. \begin{tikzpicture}[scale=0.4, baseline={([yshift=-\the\fontdimen22\textfont2]current bounding box.center)}]
	\draw (0,0) rectangle (3,3);
	\draw (0,2) -- (3,2);
	\draw (1,0) -- (1,3);
	\draw (0,1) -- (3,1);
	\draw (2,0) -- (2,3);
	
	\node (1) at (0.5, 2.5) {$a$};
	\node (2) at (1.5, 2.5) {$b$};
	\node (3) at (2.5, 2.5) {$c$};
	\node (4) at (0.5, 1.5) {$0$};
	\node (5) at (1.5, 1.5) {$d$};
	\node (6) at (2.5,1.5) {$e$};
	\node (7) at (0.5, 0.5) {$0$};
	\node (8) at (1.5, 0.5) {$0$};
	\node (9) at (2.5, 0.5) {$f$};
\end{tikzpicture} \, \right|\, b=c=0 \text{ if }i=1, e = 0 \text{ if } i=2, \text{ no restriction if } i=3\right\}.
\]
So the index = orbit length of $e_{i}$ is $q^{2}$, $q$, $1$ for $i=1,2,3$. As the factors of automorphy are trivial in this case, we find
\begin{equation}
	P(\alpha_{2})(e) = 0, q^{2}-q, 0
\end{equation}
if $e$ is $\bar{\Gamma}_{\mathbf{r}}$-equivalent with $e_{i}$. Therefore
\begin{equation} \label{Eq.Case-r-sum-of-van-der-Puts}
	\sum_{e \in \mathbf{A}_{\mathbf{r},1}} P(\alpha_{2})(e) = q(q^{2}-q) = q^{3} - q^{2}.
\end{equation}
For $\boldsymbol{\omega} \in \mathbf{F}_{\mathbf{r}}$, the characteristic sequence of $\Lambda_{\boldsymbol{\omega}}$ starts with $\lambda_{1} = 1$, $\lambda_{2} = T$, $\lambda_{3} = \omega_{2}$,
while $\lvert \lambda_{4} \rvert > \lvert \lambda_{3} \rvert = q$. As in \eqref{Eq.Alpha-2-of-omega} ,
\begin{equation}
	\alpha_{2}(\boldsymbol{\omega}) \equiv \alpha_{2}(W')
\end{equation}
with $W' = \mathds{F} + \mathds{F}T + \mathds{F}\omega_{2}$. By the same argument as in \eqref{Eq.Case-p-Local-inner-degree}  we find
\begin{equation}
	N_{\mathbf{r}}(\alpha_{2}) = q N_{\mathbf{0}}(\beta),
\end{equation}
where $\beta$ is the function $\beta(\boldsymbol{\omega}) = \alpha_{1}(\mathds{F}\omega_{2} + \mathds{F})$. Now the zeroes of $\beta(\omega_{1}:\omega_{2}:1)$ in $\mathds{P}^{2}/\mathds{F}$ are the $q^{2}-q$
lines $\{ (*:\omega_{2}:1) \}$ with $\omega_{2} \in \mathds{F}^{(2)} \smallsetminus \mathds{F}$ (see \cite{Gekeler1999-2} Proposition 8.4; $\mathds{F}^{(2)}$ is the quadratic extension of $\mathds{F}$ in
$\bar{\mathds{F}}$), each with multiplicity 1. Therefore $N_{\boldsymbol{0}}(\beta) = q^{2}-q$ and
\begin{equation}
	N_{\mathbf{r}}(\alpha_{2}) = q(q^{2}-q) = q^{3} - q^{2}
\end{equation}
as predicted by \eqref{Eq.Case-r-sum-of-van-der-Puts} .

\subsection{} \textbf{Case} (s) of \ref{Sub.Case-distinction-o-p-q-r-s} , $\mathbf{s} = (n_{1}, n_{2})$ with $n_{1} \geq n_{2} > 1$. Here both $P(\alpha_{2})(e) = 0$ for each $e \in \mathbf{A}_{\mathbf{s},1}$ by \eqref{Eq.Log-of-spectral-norm-of-alpha-2} and
$N_{\mathbf{s}}(\alpha_{2}) = 0$ since $\mathbf{s} \notin \mathcal{W}(2)$.

\textbf{Conclusion:} The preceding discussion yields complete control over the spectral norm function
\begin{align*}
	\log_{q} \lVert \alpha_{2} \rVert \colon \mathcal{BT}^{3}(\mathds{Q})	&\longrightarrow \mathds{Q}, \\
																								\mathbf{x}							&\longmapsto \log_{q} \lVert \alpha_{2} \rVert_{\mathbf{x}}
\end{align*}
the zero locus $\mathcal{BT}^{3}(\alpha_{2}) = \lambda( \Omega^{3}(\alpha_{2}))$ of $\alpha_{2}$ and the reductions of $\alpha_{2}$ both at vertices in and not in $\mathcal{BT}^{3}(\alpha_{2})$. It may
be performed with essentially the same ideas, but increasing complexity, for $\alpha_{k}$ with larger $k$ and $r$. It is desirable to dispose of similar investigations of the forms ${}_{a}\ell_{k}$ and notably
of the basic forms $g_{k} = {}_{T}\ell_{k}$.

\end{document}